\documentclass[a4paper,10pt]{amsart}

\usepackage{amsmath}
\usepackage{amssymb}
\usepackage{amsthm}
\usepackage{verbatim}
\usepackage[all]{xy}
\usepackage{enumerate}

\usepackage{pgfgantt}

\usetikzlibrary{shapes.arrows}

\usepackage{color}

\newtheorem{thm}{Theorem}[section]
\newtheorem{prop}[thm]{Proposition}
\newtheorem{cor}[thm]{Corollary}
\newtheorem{lem}[thm]{Lemma}

\theoremstyle{definition}

\newtheorem{dfn}[thm]{Definition}

\newtheorem{rmk}[thm]{Remark}

\numberwithin{equation}{section}
\newcommand{\bG}{\mathbb{G}}

\newcommand{\id}{\textrm{id}}

\newcommand{\Irr}{{\rm Irr}}
\newcommand{\Pol}{{\rm Pol}}

\newcommand{\Rep}{{\rm Rep}}

\newcommand{\cS}{\mathcal{S}}

\newcommand{\bH}{\mathbb{H}}

\newcommand{\qdim}{\: {\rm qdim}}

\newcommand{\Mor}{ {\rm Mor} }

 \title[Riesz transforms on compact quantum groups and strong solidity]{Riesz transforms on compact quantum groups and strong solidity}

\date{\noindent \today.  
MC is supported by the NWO Vidi grant ‘Noncommutative harmonic analysis and rigidity of operator algebras’, VI.Vidi.192.018.
 }

\author[Martijn Caspers]{Martijn Caspers}
\address{TU Delft, EWI/DIAM,
	P.O.Box 5031,
	2600 GA Delft,
	The Netherlands}
\email{m.p.t.caspers@tudelft.nl}

\begin{document}

\maketitle

\begin{abstract}
One of the main aims of this paper is to give a large class of strongly solid compact quantum groups. We do this by using quantum Markov semi-groups (QMS's)  and non-commutative Riesz transforms. We introduce a property for QMS's of central multipliers on a compact quantum group which we shall call `approximate linearity with almost commuting intertwiners'.   We show that this property is stable under free products, monoidal equivalence, free wreath products and  dual quantum subgroups. Examples include in particular all the (higher dimensional) free orthogonal easy quantum groups.

We then show that a compact quantum group with a QMS that is approximately linear with almost commuting intertwiners, satisfies the immediately gradient-$\cS_2$ condition from \cite{CaspersGradient} and derive strong solidity results (following \cite{CaspersGradient}, \cite{OzawaPopaAJM}, \cite{Peterson}). Using the non-commutative Riesz transform  we also show that these quantum groups have the Akemann-Ostrand property; in particular the same strong solidity results follow again (now following \cite{IsonoTrams}, \cite{PopaVaesCrelle}).
 
\end{abstract}

In their fundamental papers Voiculescu \cite{Voiculescu} and  Ozawa--Popa \cite{OzawaPopaAnnals}  prove that the free group factors $\mathcal{L}(\mathbb{F}_n), n \geq 2$ do not contain a Cartan subalgebra. This means that $\mathcal{L}(\mathbb{F}_n)$ does not contain a maximal abelian von Neumann subalgebra whose normalizer generates $\mathcal{L}(\mathbb{F}_n)$.  Consequently $\mathcal{L}(\mathbb{F}_n)$ does not admit a natural crossed product decomposition and is therefore distinguishable from the class of group measure space von Neumann algebras. The proof of Ozawa and Popa in fact shows a stronger property. They show that the normalizer of any diffuse amenable von Neumann subalgebra of $\mathcal{L}(\mathbb{F}_n)$ generates a von Neumann algebra that is amenable again. This property became known as strong solidity. After \cite{OzawaPopaAnnals} many von Neumann algebras were proven to be strongly solid.

These strong solidity results required several techniques that come from approximation properties and the geometry of groups. The proof of Ozawa-Popa \cite{OzawaPopaAnnals} essentially splits into two parts. Firstly, they show that weak amenablity of a group (or the W$^\ast$-CCAP  of its von Neumann algebra) can be used to prove a so-called weak compactness property. Secondly, using weak compactness and Popa's deformation and spectral gap techniques they obtain their aforementioned results. For the second part a number of alternative approaches have been presented. Essentially they split into three methods. (1) The use of malleable deformations \cite{OzawaPopaAnnals}. (2) The use of closable derivations in 1-cohomology and HH$^+$-type properties \cite{OzawaPopaAnnals}, \cite{Peterson}. (3) The use of the Akemann-Ostrand property \cite{PopaVaesCrelle} or quasi-cohomological methods \cite{ChifanSinclairUdrea}. (2) and (3) are closely related (see also \cite{CIW} and Section \ref{Sect=AO}). Each of these approaches provide new classes of von Neumann algebras that are strongly solid.

\vspace{0.3cm}

 We believe it is instructive to include the following diagram at this point, since these global methods shall not appear very explicitly in this paper (but rather in the references). Our focus here is  to show that the input for (2) and (3) can be proved for a reasonably large class of quantum groups. We shall thus concentrate on the bold face part of the diagram on which we expound below. The arrows should not always be understood as strict implications; sometimes additional conditions are needed.

\hyphenation{li-nea-ri-ty}
\hyphenation{imme-dia-te-ly}

\begin{figure}[ht]
	\centering
	\begin{tikzpicture}[> = latex, thick,
	bx/.style = {draw, black},
	qbx/.style = {bx, inner sep = 20pt},
	rbx/.style = {bx, inner sep = 10pt, rounded corners = 8pt}
	]
	
	\def\w{4}

	\draw (0,-4.5) -- (3.6,-4.5) -- (3.6,-3.5) -- (0,-3.5) -- (0,-4.5);	
	\node[text width=4cm] at (2.2,-4) { (3) Akemann-\\Ostrand };

	\draw (0,-3) -- (3.6,-3) -- (3.6,-2) -- (0,-2) -- (0,-3);	
	\node[text width=4cm] at (2.2,-2.5) {(2) Derivations  };

	\draw (0,-1.5) -- (3.6,-1.5) -- (3.6,-0.5) -- (0,-0.5) -- (0,-1.5);	
	\node[text width=4cm] at (2.2,-1) {(1) Malleable\\ deformations };		

	\draw (-4.6,-3) -- (-0.4,-3) -- (-0.4,-2) -- (-4.6,-2) -- (-4.6,-3);	
	\node[text width=4cm] at (-2.4,-2.5) { {\bf QMS's and \\ Gradient-$\cS_2$  }};

 	\draw (-9.6,-3.3) -- (-5.2,-3.3) -- (-5.2,-1.7) -- (-9.6,-1.7) -- (-9.6,-3.3);	
	\node[text width=4cm] at (-7.4,-2.5) {  {\bf Approximate linear \\ + almost commuting \\ intertwiners } };

	\node[single arrow,draw=black,fill=black!10,minimum height=0.6cm, minimum width = 0.1cm, rotate=0] at (-5,-2.5) {$\:$};
	\node[single arrow,draw=black,fill=black!10,minimum height=0.6cm, minimum width = 0.1cm, rotate=0] at (-0.2,-2.5) {$\:$};
	\node[single arrow,draw=black,fill=black!10,minimum height=1.2cm, minimum width = 0.1cm, rotate=-45] at (-0.2,-3.2) {$\:$};
	\node[single arrow,draw=black,fill=black!10,minimum height=1cm, minimum width = 0.1cm, rotate=0] at (4.0,-2.5) {$\:$};
	\node[single arrow,draw=black,fill=black!10,minimum height=1.4cm, minimum width = 0.1cm, rotate=45] at (4,-3.7) {$\:$};
	\node[single arrow,draw=black,fill=black!10,minimum height=1.4cm, minimum width = 0.1cm, rotate=-45] at (4,-1.3) {$\:$};

    \draw (-8,-6.5) -- (-4.4,-6.5) -- (-4.4,-5.5) -- (-8,-5.5) -- (-8,-6.5);	
	\node[text width=4cm] at (-5.8,-6) { W$^\ast$-CCAP  \\or W$^\ast$-CBAP };

    \draw (-3,-6.5) -- (0.6,-6.5) -- (0.6,-5.5) -- (-3,-5.5) -- (-3,-6.5);	
	\node[text width=4cm] at (-0.8,-6) { Weak compactness };

	\draw (2,-8) -- (5.6,-8) -- (5.6,-7) -- (2,-7) -- (2,-8);	
	\node[text width=4cm] at (4.2,-7.5) { Strong solidity };

	\draw (4.6,-3) -- (5.6,-3) -- (5.6,-2) -- (4.6,-2) -- (4.6,-3);	
	\node[text width=1cm] at (5.3,-2.5) { {\it Or} };

	\draw (4.6,-6.5) -- (5.6,-6.5) -- (5.6,-5.5) -- (4.6,-5.5) -- (4.6,-6.5);	
	\node[text width=1cm] at (5.2,-6) { {\it And} };

	\node[single arrow,draw=black,fill=black!10,minimum height=2.5cm, minimum width = 0.1cm, rotate=-90] at (5.1,-4.1) {$\:$};

	\node[single arrow,draw=black,fill=black!10,minimum height=0.7cm, minimum width = 0.1cm, rotate=-90] at (5.1,-6.7) {$\:$};

	\node[single arrow,draw=black,fill=black!10,minimum height=1.4cm, minimum width = 0.1cm, rotate=0] at (-3.8,-6) {$\:$};

	\node[single arrow,draw=black,fill=black!10,minimum height=4.1cm, minimum width = 0.1cm, rotate=0] at (2.5,-6) {$\:$};

	
	\end{tikzpicture}
\end{figure}

\vspace{0.3cm}

In \cite{IsonoTrams}, \cite{IsonoIMRN} Isono provided the first  examples of von Neumann algebras coming from the theory of compact quantum groups that are strongly solid. The approach falls in category (3) described above. In particular Isono proves that free orthogonal quantum groups are strongly solid. Later different proofs of this fact were given in \cite{FimaVergnioux} and \cite{CaspersGradient} (see also the earlier paper \cite{VaesVergnioux} on solidity). In \cite{BrannanDocumenta} and \cite{IsonoIMRN}  strong solidity results for quantum automorphism groups have been obtained.

We note that in \cite[Theorem C]{IsonoIMRN} also free products of free orthogonal/unitary quantum groups and quantum automorphism groups  are covered. In the current paper we shall deal with a property that implies strong solidity and which is stable under free products and monoidal equivalence. One advantage of this approach is that our methods apply to a free product of (certain) compact quantum groups followed by a monoidal equivalence. This is especially important for the treatment of free wreath products \cite{Bichon}, \cite{LemeuxTarrago}.

In \cite{CaspersGradient} it was proved that also the type III deformations of free orthogonal and unitary quantum groups are strongly solid. The proof builds upon the weak compactness properties from \cite{BHV} and follows the path of (2) described above. The theory of quantum Markov semi-groups (QMS's) is used to construct the closable derivations in (2) from \cite{CiprianiSauvageot}. This is done for the specific examples of free orthogonal and unitary quantum groups.

\vspace{0.3cm}

This paper continues the line of \cite{CaspersGradient} by involving two new ideas. Firstly, we look at \cite{CaspersGradient} from the viewpoint of a rigid C$^\ast$-tensor category. Though that this paper is not written in the abstract language of C$^\ast$-tensor categories (as we found this more accessible), this is precisely the structure of $\Irr(\bG)$ that occurs in our proofs.

Secondly, we refine the method from \cite{CaspersGradient}.
We introduce a new property for a QMS of central multipliers on a compact quantum group which we call `approximate linearity with almost commuting intertwiners', see Definition \ref{Dfn=Almost}. The definition is certainly technical in nature, but it has some clear advantages. Namely, it is immediately clear that it is invariant under monoidal equivalence of quantum groups. A first consequence is that since the free orthogonal quantum groups $O_N^+$ are monoidally equivalent to $SU_q(2), q \in (0,1)$ with $q + q^{-1} = N$  the estimates from \cite{CaspersGradient} can be carried out on $SU_q(2)$. We also prove a couple of other stability properties, including free wreath products.

\begin{thm}\label{Thm=IntroStable}
Approximate linearity with almost commuting intertwiners of a QMS of central multipliers is stable under:
\begin{enumerate}
\item Monoidal equivalence.
\item Free products.
\item Taking dual quantum subgroups.
\item Free wreath products with $S_N^+$ (more precisely, Theorem \ref{Thm=WreathQMSdot}).
\end{enumerate}
\end{thm}

The proof for free wreath products is a combination of \cite[Theorem 5.11]{LemeuxTarrago}, the other stability properties and the fact that $SU_q(2)$ carries a QMS that is approximately linear with almost commuting intertwiners. To prove the latter statements we provide a conceptual way to construct QMS's from suitable families of ucp maps. This makes use of generating functionals and differentiation at 0. The proof also simplifies \cite[Section 6.1]{CaspersGradient}. The author is indepted to Adam Skalski for sharing this argument.

We then show that indeed the strong solidity and Akemann-Ostrand type results as in the diagram above are implied. We first show the following (following the path (2)).

\begin{thm}\label{Thm=IntroSolid}
Let $\bG$ be a  compact quantum group of Kac type such that $L_\infty(\bG)$ has the W$^\ast$-CBAP. Suppose that $\bG$ carries a QMS of central multipliers that is approximately linear with almost commuting intertwiners and which is immediately $L_2$-compact. Then $L_\infty(\bG)$ is strongly solid.
\end{thm}

Then we show the following theorem using non-commutative Riesz transforms (see also \cite{CIW}). Since the Akemann-Ostrand property could be of independent interest we record it in this paper in a separate section.

\begin{thm}\label{Thm=IntroAO}
Let $\bG$ be a  compact quantum group of Kac type such that $C_r(\bG)$ is locally reflexive. Suppose that $\bG$ carries a QMS of central multipliers that is approximately linear with almost commuting intertwiners and which is immediately $L_2$-compact. Then $L_\infty(\bG)$ satisfies the Akemann-Ostrand property (more precisely AO$^{+}$ from \cite{IsonoTrams}).
\end{thm}

In \cite{IsonoTrams} it was proved in the factorial case that together with the W$^\ast$-CBAP Theorem \ref{Thm=IntroAO} implies strong solidity. So in that   case Theorem \ref{Thm=IntroAO} implies Theorem \ref{Thm=IntroSolid}.

\vspace{0.3cm}

We now turn to the examples. Most of the work is contained in the following theorem from which a diversity of results follow by stability properties. Its proof heavily uses the estimates \cite[Appendix]{VaesVergnioux}; it is interesting that these estimates are precisely sharp enough for our purposes.

\begin{thm}\label{Thm=IntroSUq2}
$SU_q(2)$ carries a QMS of central multipliers that is approximately linear with almost commuting multipliers and immediately $L_2$-compact.
\end{thm}

We can now harvest our results using the stability properties and several monoidal equivalence and isomorphism results for compact quantum groups that have been proved by others, most notably \cite{Raum}, \cite{RaumWeber}, \cite{Bichon}, \cite{BichonRijdtVaes}, \cite{LemeuxTarrago}.

\begin{thm}\label{Thm=IntroExample}
The following (Kac type) compact quantum groups are strongly solid and satisfy AO$^+$:
\begin{enumerate}
\item\label{Item=ClassI} All 7 series free orthogonal easy quantum groups classified in \cite{WeberAdvances}, \cite{BanicaSpeicher} under the names  $O_{N_3}^+, S_{N_5}^+, H_{N_5}^+, B_{N_4}^+, S_{N_5}'^+, B_{N_4}'^{+}$ and $B_{N_4}^{\# +}$  for $N_3 \geq 3, N_4 \geq 4, N_5  \geq 5$ (see \cite{BrannanDocumenta}, \cite{IsonoIMRN}, \cite{IsonoTrams},  \cite{FimaVergnioux}, \cite{CaspersGradient}).
\item The quantum reflection groups $H_N^{s+} \simeq \widehat{\mathbb{Z}_s} \wr_\ast S_N^+$ for $N \geq 5, \infty \geq s \geq 2$ where $\mathbb{Z}_\infty = \mathbb{Z}$.
\item The free unitary quantum groups $U_N^+$ for $N \geq 3$ (see \cite{IsonoIMRN}, \cite{IsonoTrams}, \cite{CaspersGradient}).
\end{enumerate}
\end{thm}

The selection of examples presented in Theorem \ref{Thm=IntroExample} is a bit random and not exhaustive.  We have chosen to present examples that relate to attempts  to classify easy quantum groups. The representation category of the families in Theorem \ref{Thm=IntroExample} \eqref{Item=ClassI} are precisely the ones whose representation categories can be described in terms of non-colored, non-crossing partitions. One may wonder what happens in case more colors are added to the partitions like in \cite{FreslonTrams}, \cite{TarragoWeber}. Our theorem shows that already some cases are covered.

It should be mentioned that part of Theorem \ref{Thm=IntroExample} was proved in the literature already using different methods and we have given references in the theorem. Our method gives a unified way that treats all examples at once.  To the knowledge of the author strong solidity for $H_N^+$ and the more general quantum reflection groups has not been covered and neither is  AO$^+$.  Other new examples are for instance all free wreath products of these examples with $S_N^+$.

\vspace{0.3cm}

\noindent {\it Structure.} Section \ref{Sect=Prelim} introduces preliminary notation. In Section \ref{Sect=Almost} we introduce almost linearity with almost commuting intertwiners and show stability properties. We conclude most of Theorem \ref{Thm=IntroStable} except for the wreath products.
Section \ref{Sect=Solid} contains the implications for strong solidity and proves Theorem \ref{Thm=IntroSolid}.
 In Section \ref{Sect=QMS} we show that $SU_q(2)$ carries a good QMS and prove Theorem \ref{Thm=IntroSUq2}. From this we can conclude the proof of the wreath product case in Theorem \ref{Thm=IntroStable}  as well as strong solidity of the examples of Theorem \ref{Thm=IntroExample}. This is done in Section \ref{Sect=Consequences}. In Section \ref{Sect=AO} we prove the corresponding statements for the Akemann-Ostrand property. This concludes Theorem \ref{Thm=IntroAO}.

\vspace{0.3cm}

\noindent {\bf Acknowledgements.} The author wishes to express his gratitude to Amaury Freslon, Tao Mei, Adam Skalski, Mateusz Wasilewski and Moritz Weber for their comments and/or useful discussions  that led to this paper. The author also thanks Gerrit Vos and the referee for some final improvements of the text.

\vspace{0.3cm}

\section{Preliminaries}\label{Sect=Prelim}

$\delta(x \in X)$ is the function that is 1 if $x \in X$ and 0 otherwise. Inner products are linear in the left leg. For $\xi, \eta$ vectors in a Hilbert space $H$ we write $\omega_{\xi, \eta}(x) = \langle x \xi, \eta \rangle$.

Standard theory of von Neumann algebras can be found in \cite{TakI}, \cite{TakII}. For operator spaces we refer to \cite{EffrosRuan}, \cite{Pisier}. The abbreviation ucp stands for unital completely positive.

\subsection{Finite dimensional approximations and strong solidity} See \cite{BrownOzawa} for the following notions.

\begin{dfn}
 We say that a von Neumann algebra $M$ has the {\it W$^\ast$-CBAP} (W$^\ast$ completely bounded approximation property) if there exists a  net $(\Phi_i)_i$ of normal completely bounded finite rank maps $M \rightarrow M$ such that:
\begin{enumerate}
\item There exists $\Lambda \geq 1$ such that for all $i$ we have $\Vert \Phi_i \Vert_{cb} \leq \Lambda$;
\item For every $x \in M$ we have $\Phi_i(x) \rightarrow x$ $\sigma$-weakly.
\end{enumerate}
$\Lambda$ is called the Cowling-Haagerup constant.
If $\Lambda = 1$ then we say that $M$ has the {\it W$^\ast$-CCAP} (W$^\ast$ completely contractive approximation property).
\end{dfn}

For quantum groups of Kac type the  W$^\ast$-CBAP (resp.  W$^\ast$-CCAP) is equivalent to weak amenability of the quantum group (resp. weak amenability with Cowling-Haagerup constant 1). For the Haagerup property, see also \cite{CaspersSkalskiIMRN}, \cite{OkayasuTomatsu}.

\begin{dfn}
We say that a finite von Neumann algebra with faithful normal state $(M, \tau)$ has the {\it Haagerup property} if there exists a net $(\Phi_i)_i$ of normal ucp maps $M \rightarrow M$ such that $\tau \circ \Phi_i = \tau$, such that $\Phi_i$ is compact as a map $L_2(M, \tau) \rightarrow L_2(M, \tau)$ and such that for every $x \in M$ we have $\Phi_i(x) \rightarrow x$ strongly.
\end{dfn}

We further need the notions of solidity (see \cite{BrownOzawa}, \cite{OzawaActa}) and strong solidity as in the next definition.

\begin{dfn}
A finite von Neumann algebra $M$ is called {\it strongly solid} if for every diffuse amenable von Neumann subalgebra $P \subseteq M$ we have that ${\sf Nor}_M(P)''$ is amenable, where the normalizer is defined as
\[
{\sf Nor}_M(P) = \{ u \in M \mid u \textrm{ unitary s.t. } u P u^\ast = P \}.
\]
\end{dfn}


\subsection{Compact quantum groups and represenations} The theory of compact quantum groups has been established by Woronowicz \cite{Woronowicz}.

\begin{dfn}
A compact quantum group $\bG$ is a pair $(C(\bG), \Delta_{\bG})$ of a unital C$^\ast$-algebra $C(\bG)$ and a unital $\ast$-homomorphism $\Delta_\bG: C(\bG) \rightarrow C(\bG) \otimes_{{\rm min}} C(\bG)$ (the comultiplication) satisfying $(\Delta_\bG \otimes \id) \circ \Delta_\bG = (\id \otimes \Delta_\bG) \circ \Delta_\bG$ (coassociativity) and such that both $\Delta_\bG(C(\bG))  (C(\bG) \otimes 1)$ and $\Delta_\bG(C(\bG)) (1 \otimes C(\bG))$ are dense in $C(\bG) \otimes_{{\rm min}} C(\bG)$.
\end{dfn}

A compact quantum group $\bG$ admits a unique state $\varphi$ on $C(\bG)$ called the {\it Haar state} which satisfies left and right invariance
\[
(\varphi \otimes \id) \circ \Delta_\bG(x) = \varphi(x) 1 = (\id \otimes \varphi  ) \circ \Delta_\bG(x).
\]
$\bG$ is called {\it Kac} if $\tau$ is tracial.
We let $C_r(\bG) = \pi_\varphi(C(\bG))$ and $L_\infty(\bG) = \pi_\varphi(C(\bG))''$ be the C$^\ast$-algebra and von Neumann algebra generated by the GNS-representation $\pi_\varphi$ of $\varphi$.  A (finite dimensional unitary) representation of $\bG$ is a unitary element $u \in C(\bG) \otimes M_n(\mathbb{C})$ such that $(\Delta_\bG \otimes \id)(u) = u_{13} u_{23}$ where $u_{23} = 1 \otimes u$ and $u_{13}$ is $u_{23}$ with the flip map applied to its first two tensor legs. We also set $u_{12} = u \otimes 1_n$. All representations are assumed to be unitary and finite dimensional and we shall just call them {\it representations}. The elements $(\id \otimes \omega)(y)$ with $\omega \in M_n(\mathbb{C})^\ast$ are called the {\it matrix coefficients} of $u$. We shall use Woronowicz quantum Peter-Weyl theorem \cite{Woronowicz} which states that for every $\alpha, \beta \in \Irr(\bG)$ there exists positive $Q_\alpha  \in M_{n_\alpha}(\mathbb{C})$ with  $\qdim(\alpha) := {\rm Tr}(Q_\alpha) = {\rm Tr}(Q_\alpha^{-1})$ such that
\begin{equation}\label{Eqn=PeterWeyl}
\varphi(  (u^{\beta}_{\mu, \nu})^\ast    u^{\alpha}_{\xi, \eta} ) = \delta_{\alpha, \beta} \qdim(\alpha)^{-1} \langle Q_\alpha^{\frac{1}{2}} \xi, Q_\alpha^{\frac{1}{2}} \mu \rangle \langle \xi, \eta \rangle, \qquad \xi, \eta, \mu, \nu \in \mathbb{C}^{n_\alpha}.
\end{equation}
The quantity $\qdim(\alpha)$ is called the {\it quantum dimension}.

  After these preliminaries the comultiplication $\Delta_\bG$ shall never be used and we stress that all occurrences of the greek letter $\Delta$ (without subscript $\bG$) concern generators of quantum Markov semi-groups.

Let $u^1 \in C(\bG) \otimes M_{n_1}(\mathbb{C})$ and $u^2 \in C(\bG) \otimes M_{n_2}(\mathbb{C})$. The tensor product $u^1 \otimes u^2$ is defined as the representation $u^1_{12} u^2_{13}$.
 $u$ is called irreducible if the matrix algebra generated by $(\omega \otimes \id)(u), \omega \in C(\bG)^\ast$ is simple. A morphism between $u^1 \in C(\bG) \otimes M_{n_1}(\mathbb{C})$ and $u^2 \in C(\bG) \otimes M_{n_2}(\mathbb{C})$ is a map $T: \mathbb{C}^{n_1} \rightarrow \mathbb{C}^{n_2}$ such that $u^1 (1 \otimes T) = (1 \otimes T) u^2$. Let $\Mor(u^1, u^2)$ be the (normed) vector space of morphisms. There is a quantum version of Schur's lemma that states that $u$ is irreducible if and only if $\Mor(u, u) = \mathbb{C} 1$. If $\Mor(u^1, u^2)$ contains a unitary element then $u^1$ and $u^2$ are called equivalent. We write $\Irr(\bG)$ for the equivalence classes of irreducible representations and $\Rep(\bG)$ for the equivalence classes of all finite dimensional representations. Its elements shall typically be denoted by $\alpha, \beta$ and $\gamma$. The dimension of $\alpha \in \Rep(\bG)$ is denoted by $n_\alpha$ and satisfies $n_\alpha \leq \qdim(\alpha)$.  Tensor products and $\Mor$ are well-defined on equivalence classes. For $\alpha, \beta \in \Irr(\bG)$ the tensor product $\alpha \otimes \beta$ is equivalent to a direct sum of irreducibles $\oplus_{\gamma \in \Irr(\bG)} m_\gamma \cdot \gamma$ where $m_\gamma \cdot \gamma = \oplus_{i=1}^{m_\gamma} \gamma$ is an $m_\gamma$-fold copy. This decomposition is unique up to equivalence and the set of all such decompositions is referred to as the fusion rules. We write $\alpha \subseteq \beta$ if $\Mor(\alpha, \beta)$ contains an isometry.
 For $\alpha \in \Rep(\bG)$ we denote by $\overline{\alpha}$ its {\it contragredient representation}.

\begin{prop}[Frobenius duality]\label{Lem=Frobenious}
For $\alpha, \beta, \gamma \in \Rep(\bG)$ we have $\Mor(\alpha, \beta \otimes \gamma) \simeq \Mor(\overline{\beta} \otimes \alpha, \gamma)$ linearly. Consequently, if $\alpha$ and $\gamma$ are irreducible then  $\alpha \subseteq \beta \otimes \gamma$  iff $\gamma \subseteq \overline{\beta} \otimes \alpha$.
\end{prop}

\begin{lem}\label{Lem=FrobConsequence1}
Let $\alpha, \gamma \in \Irr(\bG)$. There are only finitely many $\beta \in \Irr(\bG)$ such that $1 \subseteq \alpha \otimes \beta \otimes  \gamma$.
\end{lem}
\begin{proof}
If $1 \subseteq \alpha \otimes \beta \otimes  \gamma$ then by Frobenious duality we have $\beta \subseteq \overline{\alpha} \otimes \overline{\gamma}$ and there are only finitely many such $\beta$.
\end{proof}

We let $\Pol(\bG)$ be the $\ast$-algebra of matrix coefficients of (finite dimensional) representations of $\bG$. It is given by the linear span of $(\id \otimes \omega)(u)$ for all representations $u \in C(\bG) \otimes M_n(\mathbb{C})$ and $\omega \in M_n(\mathbb{C})^\ast$. There is a distinguished faithful $\ast$-homomorphism $\epsilon: \Pol(\bG) \rightarrow \mathbb{C}$ called the {\it counit} that satisfies
\[
(\epsilon \otimes \id) \circ \Delta_{\bG} = \id = (\id \otimes \epsilon) \circ \Delta_{\bG}.
\]
  $\Pol(\bG)$ carries the inner product $\langle x,y \rangle = \varphi(y^\ast x)$ and norm $\Vert x \Vert_2^2 = \langle x,x \rangle$. The completion of $\Pol(\bG)$ with respect to this norm is called $L_2(\bG)$ and may be identified with the GNS-space of $\varphi$.   For $\alpha \in \Irr(\bG)$ we let $P_\alpha: \Pol(\bG) \rightarrow \Pol(\bG)$ be the  orthogonal projection onto the matrix coefficients of $\alpha$.

For compact quantum groups $\mathbb{H}$ and $\mathbb{G}$ we say that $\mathbb{H}$ is a dual quantum subgroup of $\mathbb{G}$, notation $\widehat{\mathbb{H}} < \widehat{\mathbb{G}}$, if $L_\infty(\bH) \subseteq L_\infty(\bG)$ and the von Neumann algebraic comultiplication of $L_\infty(\bG)$ restricts to $L_\infty(\bH)$ as the comultiplication of $\bH$. In this case $\Irr(\bH) \subseteq \Irr(\bG)$ naturally and the fusion rules and morphisms of $\Irr(\bG)$ restrict to $\Irr(\bH)$ (it is a full subcategory).

A {\it central multiplier} $\Phi: L_\infty(\bG) \rightarrow L_\infty(\bG)$ is a map such for every $\alpha \in \Irr(\bG)$ there exist $\Delta_\alpha \in \mathbb{C}$ such that $\Phi(  (\id \otimes \omega)( \alpha)  ) = \Delta_\alpha (\id \otimes \omega)(\alpha)$ for all $\alpha \in \Irr(\bG)$ and $\omega \in M_{n_\alpha}(\mathbb{C})^\ast$. We refer to \cite{JNR} for more general background on multipliers.

\begin{rmk}
We have that $(\Irr(\bG), \Mor)$ with the tensor products, fusion rules and contragredients forms a rigid C$^\ast$-tensor category. A large part of this paper can directly be translated in terms of the abstract setting of rigid C$^\ast$-tensor categories. However, since our many applications are in quantum group theory our presentation follows the quantum group theoretical terminology. Recall that by Tannaka-Krein duality rigid C$^\ast$-tensor categories with specified fibre functor are always of the form  $(\Irr(\bG), \Mor)$ \cite{WoronowiczTannaka}.
\end{rmk}

\subsection{Quantum Markov semi-groups}
Let $M$ be a von Neumann algebra with a faithful normal state $\varphi$. A   {\it quantum Markov semi-group (QMS)} $\Phi = (\Phi_t)_{t \geq 0}$ is semi-group of normal unital completely positive maps $\Phi_t: M \rightarrow M$ such that for every $x \in M$ the map $t \mapsto \Phi_t(x)$ is strongly continuous. Moreover we assume that a QMS is GNS-symmetric in the sense that $\varphi(\Phi_t(x) y) = \varphi(x \Phi_t(y))$ for all $x,y \in M$. $\Phi$ is called $\varphi$-modular (or modular) if $\Phi_t \circ \sigma^\varphi_s = \sigma^\varphi_s \circ \Phi_t$ for all $t\geq 0, s \in \mathbb{R}$ where $\sigma^\varphi$ is the modular automorphism group of $\varphi$ \cite{TakII}. The QMS's occuring in this paper are QMS's of central multipliers which are always modular and GNS-symmetric. Further they are norm continuous on $\Pol(\bG)$.  It should also be stressed that the most important of our applications are for finite von Neumann algebras and $\varphi$ tracial. However in the analysis we shall also need the Haar state on $\bG_q = SU_q(2), q \in (-1,1)$ which is non-tracial even though $L_\infty(\bG_q)$ is of type I.

If $\Phi$ is a QMS of central multipliers then for every $\alpha \in \Irr(\bG)$ there exists $\Delta_\alpha \geq 0$ such that  $\Phi_t(x_\alpha) = \exp(-t \Delta_\alpha) x_\alpha$ for every matrix coefficient $x_\alpha$ of $\alpha$. The values $(\Delta_\alpha)_{\alpha \in \Irr(\bG)}$ completely determine $\Phi$. We set the generator $\Delta: \subseteq L_2(\bG) \rightarrow L_2(\bG)$ to be the closure of
\[
 \Pol(\bG) \rightarrow \Pol(\bG): x_\alpha \mapsto \Delta_\alpha x_\alpha.
\]
$\Phi$ is called {\it immediately $L_2$-compact} if $\Delta$ has compact resolvent.  $\Delta$ is closely related to the associated quantum Dirichlet form. In \cite{CaspersSkalskiCMP}, \cite{JolissaintMartin} it was proved that a (general) von Neumann algebra has Haagerup property if and only if it admits an immediately $L_2$-compact QMS.

\subsection{Free products}
To two compact quantum groups $\bG_1$ and $\bG_2$ one can associate a free product quantum group $\bG_1 \ast \bG_2$ \cite{WangFree}. It satisfies $L_\infty(\bG) = L_\infty(\bG_1) \ast L_\infty(\bG_2)$ where free products are taken with respect to the von Neumann algebraic Haar states. Its Haar state is the free product of the Haar states. Moreover it can be equipped with a natural comultiplication which shall not be used in this paper. What is relevant for us is the following proposition that describes $\Irr(\bG)$ as a fusion category.

\begin{prop}[See \cite{WangFree}, \cite{Valvekens} or Theorem 3.4 of \cite{CaspersFima}]\label{Thm=FreeProductRep}
Let $\bG_1$ and $\bG_2$ be compact quantum groups. A tensor product $\gamma_1 \otimes \ldots \otimes \gamma_n$ with $\gamma_i \in \Irr(\bG_{k_i})$ and $k_i \not = k_{i+1}$ is called reduced. All such reduced tensor products form a well-defined complete set of mutually inequivalent irreducible representations of $\bG_1 \ast \bG_2$. In other words they constitute $\Irr(\bG_1 \ast \bG_2)$.   The fusion rules are as follows for reduced tensors $\beta_1 \otimes \ldots \otimes \beta_l$ and $\gamma_1 \otimes \ldots \otimes \gamma_n$. If $\beta_l$ and $\gamma_1$ are not representations of the same quantum group, then
  \[
  \beta_1 \otimes \ldots \otimes \beta_l \cdot \gamma_1 \otimes \ldots \otimes \gamma_n =   \beta_1 \otimes \ldots \otimes \beta_l \otimes \gamma_1 \otimes \ldots \otimes \gamma_n.
 \]
   If $\beta_l$ and $\gamma_1$ are representations of the same quantum group then,
 \begin{equation}\label{Eqn=FreeProductDec}
 \begin{split}
  &(\beta_1 \otimes \ldots \otimes \beta_l) \otimes (\gamma_1 \otimes \ldots \otimes \gamma_n) \\
   = &
   \left( (\beta_1 \otimes \ldots \otimes \beta_{l-1}) \otimes  (\bigoplus_{i, \alpha_i \not = 1} \alpha_i ) \otimes  (\gamma_2 \otimes \ldots \otimes \gamma_n)
   \right)  \oplus  \left( \bigoplus_{i, \alpha_i = 1} (\beta_1 \otimes \ldots \otimes \beta_{l-1}) \otimes    (\gamma_2 \otimes \ldots \otimes \gamma_n) \right),
 \end{split}
 \end{equation}
 where $\beta_l \otimes \gamma_1 = \oplus_{i} \alpha_i$ is the decomposition of $\beta_l \otimes \gamma_1$ into irreducibles (with possible multiplicity).  Note that in \eqref{Eqn=FreeProductDec} the latter summand is not necessarily reduced but the fusion rules are hereby defined inductively.
\end{prop}

We shall use the short hand notation
\[
\gamma_1 \ldots \gamma_n = \gamma_1 \otimes \ldots \otimes \gamma_n,
\]
for a reduced word.

\subsection{Multiplicity freeness}
A compact quantum group $\bG$ is called {\bf multiplicity free} if for $\alpha, \beta, \gamma \in \Irr(\bG)$ the space $\Mor(\gamma, \alpha \otimes \beta)$ is $\leq 1$-dimensional. That is $\gamma$ occurs at most once in the decomposition of $\alpha \otimes \beta$ into irreducible representations.  In case $\bG_1$ and $\bG_2$ are multiplicity free then in \eqref{Eqn=FreeProductDec} the last summation is in fact a single summand if $\beta_{k} = \overline{\gamma_1}$ and it vanishes otherwise (it follows by Frobenius duality Lemma \ref{Lem=Frobenious} for instance). So we record that (the summation over $\alpha$ going over irreducible representations),
\begin{equation} \label{Eqn=FreeProductDecMultFree}
\begin{split}
  &(\beta_1  \ldots  \beta_k) \otimes (\gamma_1  \ldots  \gamma_n) \\
   = &
   \bigoplus_{i=1}^{L} \bigoplus_{1 \not = \alpha \subseteq \beta_{k-i+1} \otimes \gamma_i}
   (\beta_1   \ldots   \beta_{k-i})   \alpha    (\gamma_{i+1}   \ldots  \gamma_n),
 \end{split}
\end{equation}
where $L-1$ is the maximum index $i$ for which $\gamma_i = \overline{\beta_{k-i+1}}$. We note that the summands in \eqref{Eqn=FreeProductDecMultFree} are reduced. This decomposition shall be used without further reference in the rest of the paper.

\vspace{0.3cm}

\noindent {\bf Assumption:} Throughout the entire paper we assume that all compact quantum groups (e.g. $\bH, \bG, \bG_1$ and $\bG_2$) are multiplicity free.

\vspace{0.3cm}

The following result should be well-known and is easy to prove.

\begin{prop}
If $\bG_1$ and $\bG_2$ are compact quantum groups that are multiplicity free then so is $\bG_1 \ast \bG_2$. If $\widehat{\bH} < \widehat{\bG}$ and $\bG$ is multiplicity free then so is $\bH$.
\end{prop}
\begin{proof}
Suppose that we have an irreducible representation $\alpha = \alpha_1 \ldots \alpha_l$ contained in $(\beta_1  \ldots  \beta_k) \otimes (\gamma_1  \ldots  \gamma_n)$. Then by considering the length $\alpha$ must be one of the $i$-th summands in \eqref{Eqn=FreeProductDecMultFree} with $i$ satisfying  $2i = k +n -l+1$. But all those summands are mutually inequivalent by Proposition \ref{Thm=FreeProductRep} and the fact that $\bG_1$ and $\bG_2$ are multiplicity free.
\end{proof}

That $\bG$ is multiplicity free has the following consequence. For $\beta, \gamma \in \Irr(\bG)$ and $\alpha \subseteq \beta \otimes \gamma$ there exists an intertwiner
\[
V^{\beta, \gamma}_{\alpha} \in \Mor( \alpha, \beta \otimes \gamma),
\]
 that is moreover unique up to a phase factor. All expressions and proofs occuring in this paper are independent of this phase factor unless mentioned otherwise.

\subsection{Monoidal equivalence}

\begin{dfn}
Two compact quantum groups $\bG_1$ and $\bG_2$ are called {\it monoidally equivalent} if there exists a bijection $\pi: (\Irr(\bG_1), \Mor_{\bG_1} ) \rightarrow  (\Irr(\bG_2), \Mor_{\bG_2} )$ that maps the trivial representation of $\bG_1$ to the trivial representation of $\bG_2$ and which for any morphisms $S,T$ and unit $1_\alpha \in \Mor(\alpha, \alpha), \alpha \in \Rep(\bG)$ satisfies:
\[
\begin{split}
 &\pi(1_\alpha) = 1_\alpha, \qquad \pi( S \otimes T ) = \pi(S) \otimes \pi(T), \\
 & \pi(S^\ast) = \pi(S)^\ast, \qquad \pi(ST) = \pi(S) \pi(T),
\end{split}
\]
where in the latter equality we assume that $S$ and $T$ are composable. $\pi$ is then called a monoidal equivalence.
\end{dfn}

\begin{prop}\label{Prop=MonoidalTransference}
Let $\bG_1$ and $\bG_2$ be monoidally equivalent compact quantum groups so that we may identify $\Irr(\bG_1) = \Irr(\bG_2)$. Let $(\Phi_t^1)_{t \geq 0}$ be a QMS of central multipliers on $L_\infty(\bG_1)$ such that $\Phi_t^1(x_\alpha) = \exp(-t\Delta_\alpha) x_\alpha$ for every matrix coefficient $x_\alpha$ of $\alpha \in \Irr(\bG_1)$. Then there exists  a QMS of central multipliers $(\Phi_t^2)_{t \geq 0}$  on $L_\infty(\bG_2)$ such that $\Phi_t^2(x_\alpha) = \exp(-t\Delta_\alpha) x_\alpha$ for every matrix coefficient $x_\alpha$ of $\alpha \in \Irr(\bG_2)$.
\end{prop}
\begin{proof}
 The proof of this fact is the same as \cite[Proposition 6.3]{FreslonJFA} and is based on \cite[Theorems 3.9 and 6.1]{BichonRijdtVaes} together with a transference method.
\end{proof}

In a sense one could also say that a central QMS lives on the level of the rigid C$^\ast$-tensor category \cite{PopaVaesTensor}, \cite{NeshveyevYamashitaTesnor}.

\section{A rigid C$^\ast$-tensor category approach to gradient estimates}\label{Sect=Almost}

To a QMS on a tracial von Neumann algebra one can associate a canonical bimodule (in principle only defined over a dense subalgebra of $M$) which is called the gradient bimodule $H_\nabla$.
In \cite{CaspersGradient}, \cite{CIW} sufficient conditions were given to assure that $H_\nabla$ is in fact a von Neumann bimodule that is moreover quasi-contained in the coarse bimodule. In this section we provide a categorical viewpoint on the approach in \cite{CaspersGradient}. What we show is that the methods and estimates that occur in the proofs of \cite{CaspersGradient} actually live on the level of a monoidal category. In particular all computations in \cite{CaspersGradient} can be carried out on the level of $SU_q(2)$ after which they transfer to a much larger class of quantum groups. A particular feature of our current approach is that the properties we consider are stable under {\it repeated} applications of constructions like free products, wreath products,  taking dual quantum subgroups and monoidal equivalence. This should be compared to for instance \cite[Theorem C]{IsonoIMRN} where such results (and consequences for rigidity properties) were limited to free products of quantum groups in a specific class. We thus cover a richer class of quantum groups than what occurs in the literature so far. In particular this approach allows us to use the main result of \cite{LemeuxTarrago} and we cover in particular free wreath products and $H_N^+$. We prove for instance that $H_N^+$ is strongly solid.
We will come back to these results in the subsequent sections. In the current section we introduce the main technical definition of being `approximately linear with almost commuting intertwiners'  and prove that is stable under free products, monoidal equivalence and taking quantum subgroups.

\subsection{Approximately linear with almost commuting intertwiners}
Let $\bG$ be a compact quantum group and recall that it is assumed multiplicity free.  For $\alpha, \beta, \gamma \in \Irr(\bG), \beta_2 \subseteq \alpha \otimes \beta \otimes \gamma$ we define
\[
\begin{split}
L_\beta^{\alpha, \gamma} = & \{ (\beta_1, \beta_2) \in \Irr(\bG) \times  \Irr(\bG)  \mid   \beta_1 \subseteq \alpha \otimes \beta, \beta_2 \subseteq \beta_1 \otimes \gamma  \},  \\
R_\beta^{\alpha, \gamma} = & \{ (\beta_1, \beta_2)  \in \Irr(\bG) \times  \Irr(\bG)  \mid  \beta_1 \subseteq \beta \otimes \gamma, \beta_2 \subseteq \alpha \otimes \beta_1   \}, \\
L_{\beta, \beta_2}^{\alpha, \gamma} = & \{ \beta_1  \in \Irr(\bG)  \mid (\beta_1, \beta_2) \in  L_\beta^{\alpha, \gamma}   \},  \\
R_{\beta, \beta_2}^{\alpha, \gamma} = & \{ \beta_1  \in \Irr(\bG)  \mid (\beta_1, \beta_2) \in  R_\beta^{\alpha, \gamma}   \}.
\end{split}
\]

\begin{lem}\label{Lem=RepBound}
  Given $\alpha, \gamma \in \Irr(\bG)$ the number of elements in the sets $L_\beta^{\alpha, \gamma}, R_\beta^{\alpha, \gamma}, L_{\beta, \beta_2}^{\alpha, \gamma}, R_{\beta, \beta_2}^{\alpha, \gamma}$ is bounded uniformly in $\beta, \beta_2$.
\end{lem}
\begin{proof}
Suppose that $\beta_1 \subseteq \alpha \otimes \beta, \beta_1 \in \Irr(\bG)$  then by Frobenius duality Lemma \ref{Lem=Frobenious} we have that $\beta \subseteq \overline{\alpha} \otimes \beta_1$. But this can only happen if $\dim(\beta) \leq \dim(\overline{\alpha})   \dim(\beta_1)$. So that $\dim(\beta_1) \geq \dim(\beta) \dim(\overline{\alpha})^{-1}$. By counting dimensions   we see that $\alpha \otimes \beta$ can therefore have at most $\dim(\overline{\alpha})$ irreducible inequivalent subrepresentations. Applying the same argument in turn to $\beta_1 \otimes \beta \otimes \gamma$ we see that there are at most   $\dim(\overline{\gamma})$ irreducible representations contained in this representation.
\end{proof}

Let $\Phi := (\Phi_t)_{t \geq 0}$ be a QMS of central multipliers on   $\bG$. The following definition is our main technical tool. Recall that we need $\bG$ to be multiplicity free to define up to a phase factor uniquely determined intertwiners $V^{\alpha, \beta}_\gamma, \alpha, \beta, \gamma \in \Irr(\bG)$. So from this point the multiplicity freeness is being used.

\begin{dfn}\label{Dfn=Almost}
We say that $\Phi$ is {\it approximately linear with almost commuting intertwiners} if the following holds. For every $\alpha,  \gamma \in \Irr(\bG)$ there exists a finite set $A_{00}  := A_{00}(\alpha, \gamma) \subseteq \Irr(\bG)$ such that for every $\beta \in \Irr(\bG) \backslash A_{00}$ and $\beta_2 \subseteq \alpha \otimes \beta \otimes \gamma$  there exist bijections (called the {\it $v$-maps}),
\[
 v^{\alpha, \gamma}( \: \cdot \: ; \beta, \beta_2) := v(\: \cdot \: ; \beta, \beta_2): L_{\beta, \beta_2}^{\alpha, \gamma} \rightarrow
 R_{\beta, \beta_2}^{\alpha, \gamma},
\]
such that   the following holds. There exists a set $A \subseteq \Irr(\bG) \backslash A_{00}$ and a  constant $C := C(\alpha, \gamma) > 0$  such that
\begin{enumerate}
\item \label{Eqn=Item=One} For all $\beta \in A, (\beta_1, \beta_2) \in L_\beta^{\alpha, \gamma}$ we have
\begin{equation}\label{Eqn=AlmostOne}
\begin{split}
\vert \Delta_\beta - \Delta_{\beta_1} - \Delta_{v(\beta_1; \beta,  \beta_2)} + \Delta_{\beta_2} \vert \leq & C  \qdim( \beta )^{-1},
\end{split}
\end{equation}
and
\begin{equation} \label{Eqn=AlmostOneB}
\vert \Delta_\beta - \Delta_{\beta_1} \vert \leq  C.
\end{equation}
For all $\beta \in \Irr(\bG) \backslash (A \cup A_{00}), (\beta_1, \beta_2) \in L_\beta^{\alpha, \gamma}$ we have
\begin{equation}\label{Eqn=AlmostOneVanish}
\begin{split}
  \Delta_\beta - \Delta_{\beta_1} - \Delta_{v(\beta_1; \beta,  \beta_2)} + \Delta_{\beta_2}   = & 0.
\end{split}
\end{equation}

\item \label{Eqn=Item=Two} For all $\beta \in A, (\beta_1, \beta_2) \in L_\beta^{\alpha, \gamma}$ we have
\begin{equation}\label{Eqn=AlmostTwo}
\inf_{z \in \mathbb{T}}\Vert V^{\beta_1, \gamma}_{\beta_2} ( V^{\alpha, \beta}_{\beta_1} \otimes \id_\gamma ) - z V^{\alpha, v(\beta_1; \beta,   \beta_2)}_{\beta_2}  (\id_\alpha \otimes V^{\beta, \gamma}_{ v(\beta_1; \beta,  \beta_2) } ) \Vert \leq C    \qdim(\beta)^{-1}.
\end{equation}
For all $\beta \in \Irr(\bG) \backslash (A \cup A_{00}), (\beta_1, \beta_2) \in L_\beta^{\alpha, \gamma}$ we have
\begin{equation}\label{Eqn=AlmostTwoVanish}
\inf_{z \in \mathbb{T}}\Vert V^{\beta_1, \gamma}_{\beta_2} ( V^{\alpha, \beta}_{\beta_1} \otimes \id_\gamma ) - z V^{\alpha, v(\beta_1; \beta,   \beta_2)}_{\beta_2}  (\id_\alpha \otimes V^{\beta, \gamma}_{ v(\beta_1; \beta,  \beta_2) } ) \Vert =0.
\end{equation}

\item \label{Eqn=Item=Three} There exists a polynomial $P$ such that for every $N \in \mathbb{N}$ we have
\begin{equation}\label{Eqn=AlmostThree}
\# \{  \beta \in A \mid \Delta_\beta < N   \} \leq  P(N).
\end{equation}
and we have that $\beta \mapsto \delta(\beta \in A) \qdim(\beta)^{-1}$ is square summable.
\end{enumerate}
\end{dfn}

\begin{rmk}
In summary Definition \ref{Dfn=Almost} entails the following. We cut $\Irr(\bG)$ into three disjoint sets. Each of these sets has a size condition and a condition on estimates of eigenvalues of $\Delta$ as well as certain almost commutations of intertwiners:
\[
\begin{array}{|l|l|l|} \hline
A_{00}  & A & \textrm{The rest: } \Irr(\bG) \backslash (A_{00} \cup A) \\ \hline
\textrm{ -Finite set} & \textrm{-Grows polynomially compared to } \Delta & \textrm{-No size restrictions} \\
\textrm{ -No conditions} & \textrm{-Estimates } \eqref{Eqn=AlmostOne}, \eqref{Eqn=AlmostOneB} \textrm{ and } \eqref{Eqn=AlmostTwo}   & \textrm{-Vanishing of } \eqref{Eqn=AlmostOneVanish}, \eqref{Eqn=AlmostTwoVanish} \\ \hline
\end{array}
\]
We shall usually refer to property \eqref{Eqn=Item=One} as being approximately linear and \eqref{Eqn=Item=Two} as having almost commuting intertwiners. We note that they have to be satisfied for the same choice of $A$ and $A_{00}$ which is why we did not define `approximate linearity' and `almost commuting intertwiners'  as independent notions.
\end{rmk}

\begin{thm}\label{Thm=Monoidal}
The property of $\Phi$ being approximately linear with almost  commuting intertwiners is stable under monoidal equivalence of compact quantum groups.
\end{thm}
\begin{proof}
Monoidally equivalent compact quantum groups have the same representation category seen as a rigid C$^\ast$-tensor category. In particular the quantum dimension, norms of intertwiners and irreducible representations with their fusion rules are invariant under monoidal equivalence (see \cite[Remarks 3.3, 3.4 and 3.4]{BichonRijdtVaes}). Since all properties in Definition \ref{Dfn=Almost} are expressed in these terms the theorem follows directly.
\end{proof}

The following theorem is clear to specialists.  For completeness we give its proof.

\begin{thm}\label{Thm=SubGroupdot}
Suppose that $\Phi$ is a QMS of central multipliers on a compact quantum group $\bG$. Suppose that $\bH$ is a compact quantum group with $\widehat{\bH} < \widehat{\bG}$. Then  $\Irr(\bH) \subseteq \Irr(\bG)$
and $L_\infty(\bH) \subseteq L_\infty(\bG)$. In particular the restriction of $\Phi$ to $L_\infty(\bH)$ is a QMS of central multipliers.  Furthermore, if $\Phi$ is approximately linear with almost commuting intertwiners, then so is its restriction to $L_\infty(\bH)$.
\end{thm}
\begin{proof}
 Indeed, if  $\widehat{\bH} < \widehat{\bG}$ then there exists a surjective $\ast$-homomorphism $\widehat{\pi}: \ell_\infty(\widehat{\bG}) \rightarrow \ell_\infty(\widehat{\bH})$. Since  $\ell_\infty(\widehat{\bG})$ is an $\ell_\infty$-direct sum of finite dimensional simple C$^\ast$-algebras (i.e. matrix algebras) $\widehat{\pi}$ must be either 0 or faithful on each of the simple matrix blocks. Then $\ell_\infty(\widehat{\bH})$ is given by the $\ell_\infty$-direct sum of all matrix blocks for which $\widehat{\pi}$ is faithful. The matrix blocks of $\ell_\infty(\widehat{\bG})$ are labelled by $\Irr(\bG)$ and the matrix blocks of $\ell_\infty(\widehat{\bH})$ are labelled by $\Irr(\bH)$ which thus is a subset of $\Irr(\bG)$. Since $L_\infty(\bH)$ is generated by the matrix coefficients of $\Irr(\bH)$ it must thus be a subalgebra of $L_\infty(\bG)$. We see that $\Phi$ restricts to $L_\infty(\bH)$ and is again a QMS of central multipliers. It is clear that $\Phi$ restricted to $L_\infty(\bH)$ satisfies Definition \ref{Dfn=Almost}  since one has to check less conditions than for the original $\Phi$ (in particular the sets $L^{\alpha, \gamma}_{\beta, \beta_2}$ and $R^{\alpha, \gamma}_{\beta, \beta_2}$ and the bijection $v^{\alpha, \gamma}(\cdot ; \beta, \beta_2)$ stay the same but need only be considered for $\alpha, \beta, \gamma \in \Irr(\bH), \beta_2 \subseteq \alpha \otimes \beta \otimes \gamma$).
\end{proof}

\subsection{Free products}

Our next aim is to show that Definition \ref{Dfn=Almost} is stable under free products.

\begin{thm}\label{Thm=FreeProduct}
Let $\Phi^1$ and $\Phi^2$ be QMS's of central multipliers on respective compact quantum groups $\bG_1$ and $\bG_2$. Let $\Phi = \Phi^1 \ast \Phi^2$ be the free product QMS of central multipliers on $\bG_1 \ast \bG_2$. If $\Phi^1$ and $\Phi^2$ are both approximately linear with almost commuting intertwiners then so is $\Phi$.
\end{thm}

The proof of Theorem \ref{Thm=FreeProduct} will take the rest of this section for which we will fix the following notation. Firstly we let $\Delta$ be the generator of $\Phi$ with eigenvalues $\Delta_\alpha, \alpha \in \Irr(\bG)$. In particular this defines $\Delta_\alpha$ for the subsets $\Irr(\bG_1)$ and $\Irr(\bG_2)$ of $\Irr(\bG)$. The straightforward proof of the following lemma can be found at \cite[Beginning of Section 5]{CaspersGradient}.

\begin{lem}[Leibniz rule]
For $\beta = \beta_1 \ldots \beta_l \in \Irr(\bG)$ a reduced word we have
\[
\Delta_\beta = \sum_{r=1}^l \Delta_{\beta_r}.
\]
\end{lem}

 Now let
\[
\alpha = \alpha_1 \ldots \alpha_k,\qquad \gamma = \gamma_1 \ldots \gamma_m,
\]
  in $\Irr(\bG)$ be  reduced words of representations of lengths $k$ and $m$ respectively. So $\alpha_i, i = 1, \ldots, k$ is alternatingly in $\Irr(\bG_{1})$ and $\Irr(\bG_2)$ and similarly for $\gamma_i$.
In case $\alpha_i, \gamma_j \in \Irr(\bG_1)$ (or $\alpha_i, \gamma_j \in \Irr(\bG_2)$) we define $A_{00}^1(\alpha_{i}, \gamma_{j})$ and $A^1(\alpha_i, \gamma_j)$ (or $A_{00}^2(\alpha_{i}, \gamma_{j})$ and $A^2(\alpha_i, \gamma_j)$) to be the sets $A_{00}$ and $A$ of Defintion \ref{Dfn=Almost} for  $\bG_1$ (or $\bG_2$) with respect to $\alpha_i, \gamma_j$ and $\Phi^1$ (or $\Phi^2$). This makes sense because of the assumption that $\alpha_i$ and $\gamma_j$ are representations of the same quantum group.

\vspace{0.3cm}

\noindent {\it Definition of $A_{00}$ and $A$ associated to $\alpha, \gamma \in \bG$.} The set $A_{00} \subseteq \Irr(\bG)$ will consist of all representations $\beta \in \Irr(\bG)$ of the following form:
\begin{itemize}
\item $\beta$ equals a reduced word $\beta = \overline{\alpha}_k \ldots \overline{\alpha}_{k-i+1} \overline{\gamma}_{j}  \ldots \overline{\gamma}_{1}$ for some $0 \leq i \leq k, 0 \leq j \leq m$.
\item $\beta$ equals a reduced word $\beta = \overline{\alpha}_k \ldots \overline{\alpha}_{k-i+1} \beta_{i+1} \overline{\gamma}_{j}  \ldots \overline{\gamma}_{1}$ for some $0 \leq i < k, 0 \leq j < m$ and at least one of the following holds:
\begin{itemize}
\item  $\beta_{i+1} \in A_{00}^s(\alpha_{k-i}, \gamma_{j+1})$ in case there is $s \in 1,2$ such that  $\alpha_{k-i}, \gamma_{j+1} \in \Irr(\bG_s)$,
\item $1 \subseteq \alpha_{k-i} \otimes \beta_{i+1}  \otimes \gamma_{j+1}$.
\end{itemize}
\end{itemize}
Since $A_{00}^s, s=1,2$ is finite (for the first sub-bullet) and we have Lemma \ref{Lem=FrobConsequence1}  (for the second sub-bullet) we see that $A_{00}$ is a finite set. We set $A \subseteq \Irr(\bG)$ to be the set of representations $\beta \in \Irr(\bG)$ of the following form:
\begin{itemize}
\item $\beta$ equals a reduced word  $\beta = \overline{\alpha}_k \ldots \overline{\alpha}_{k-i+1} \beta_{i+1} \overline{\gamma}_{j}  \ldots \overline{\gamma}_{1}$ for some $0 \leq i < k, 0 \leq j < m$ and $\beta_{i+1} \in A^s(\alpha_{k-i}, \gamma_{j+1})$ in case there is $s \in 1,2$ such that  $\alpha_{k-i-1}, \gamma_{j+1} \in \Irr(\bG_s)$.
\end{itemize}

\vspace{0.3cm}

As part of the proof of Theorem \ref{Thm=FreeProduct} we shall at this point already establish that Property \eqref{Eqn=Item=Three} of Definiton \ref{Dfn=Almost} holds.
\begin{lem}
 Property \eqref{Eqn=Item=Three} holds for $\bG$ and the above choice of $A$.
\end{lem}
\begin{proof}
The QMS's on $\bG_1$ and $\bG_2$ are both approximately linear with almost commuting intertwiners. Therefore let $P$ be a polynomial such that for all  possible choices $s = 1,2$ and $1 \leq i \leq k, 1 \leq j \leq m$ such that $\alpha_i, \gamma_j \in \Irr(\bG_s)$   we have for all $N \in \mathbb{N}$ that
\[
\# \{ \widetilde{\beta} \in A^s(\alpha_i, \gamma_j) \mid \Delta_{\widetilde{\beta}} \leq N \} \leq P(N).
\]
Suppose that $\beta \in A$. Then from the definition of $A$ we see that the length of the reduced expression $\beta = \beta_1 \ldots \beta_l$ cannot be longer than the sum of the lengths of $\alpha$ and $\gamma$ minus 1, i.e. $l \leq k +m -1$. Moreover we may write  $\beta = \beta_1 \ldots \beta_l =  \overline{\alpha}_k \ldots \overline{\alpha}_{k-i+1} \beta_{i+1} \overline{\gamma}_{j}  \ldots \overline{\gamma}_{1}$ for some $0 \leq i < k, 0 \leq j < m$ with $i+j+2=l$ and there is an $s = 1,2$ such that $\beta_{i+1} \in A^s(\alpha_{k-i}, \gamma_{j+1})$.
 We have by the Leibniz rule $\Delta_\beta := \Delta_{\beta_1 \ldots \beta_l} =  \sum_{r=1}^l \Delta_{\beta_r}$.
  If $\Delta_\beta \leq N$ then certainly   $\Delta_{\beta_{i+1}} \leq N$. Therefore, we crudely estimate,
 \[
 \# \{ \beta \in A \mid \Delta_\beta \leq N \} \leq (k+m-1)^2 P(N).
 \]
 This concludes the proof of the growth bound on $A$ as in \eqref{Eqn=Item=Three} of Definition \ref{Dfn=Almost}. From a similar reasoning it also follows that $\beta \mapsto \delta(\beta) \qdim(\beta)^{-1}$ is square summable.
 \end{proof}

\noindent {\it Definition of the bijections $v^{\alpha, \beta}(\: \cdot \: ; \beta, \beta_2)$ for $\bG$.}  Take $\beta \in \Irr(\bG) \backslash A_{00}$. There are three cases to be treated.

\vspace{0.3cm}

\noindent {\it Case 1.}
Assume that there exists some $i < j$ such that we have a decomposition as a reduced word
\[
\beta = (\beta_1 \ldots \beta_i) (\beta_{i+1} \ldots \beta_{j-1}) (\beta_j \ldots \beta_l),
\]
where $1 \leq i$ is the smallest index for which $\beta_i$ is not the conjugate of $\alpha_{k-i+1}$ (and if this does not exist then $i=1$) and $j\leq l$ is the largest index such that $\beta_{j}$ is not the conjugate of $\gamma_{l-j+1}$  (and if this does not exist then $j=l$).
Heuristically this means that in $\alpha \otimes \beta \otimes \gamma$ the letters of $\alpha$ can annihilate at most the first $i-1$ letters of $\beta$ and that the letters of $\gamma$ can annihilate at most the last $l-j$ letters of $\beta$. More precisely, we get the following.
 The irreducible representations contained in $\alpha \otimes \beta \otimes  \gamma$ are precisely given by representations that have a reduced expression
 \[
 \beta'  (\beta_{i+1} \ldots \beta_{j-1}) \beta'' \qquad \textrm{ with } \beta'  \subseteq \alpha \otimes (\beta_1 \ldots \beta_i), \beta'' \subseteq
(\beta_j \ldots  \beta_l) \otimes \gamma
\]
 irreducible. Furthermore, we have singleton sets,
\[
\begin{split}
L^{\alpha, \gamma}_{\beta, \beta'  (\beta_{i+1} \ldots \beta_{j-1}) \beta''} =  \left\{    \beta'  (\beta_{i+1} \ldots \beta_{l})    \right\}, \qquad
R^{\alpha, \gamma}_{\beta, \beta'  (\beta_{i+1} \ldots \beta_{j-1}) \beta''} =  \left\{       (\beta_{1} \ldots \beta_{j-1}) \beta''  \right\}.
\end{split}
\]
  We therefore set the bijection from  $L^{\alpha, \gamma}_{\beta, \beta'  (\beta_{i+1} \ldots \beta_{j-1}) \beta''}$ to $R^{\alpha, \gamma}_{\beta, \beta'  (\beta_{i+1} \ldots \beta_{j-1}) \beta''}$ by
\[
v(   \beta'  (\beta_{i+1} \ldots  \beta_l)   ;   \beta , \beta'  (\beta_{i+1} \ldots \beta_{j-1}) \beta'' )
=
(\beta_1  \ldots \beta_{j-1}) \beta''.
\]

\vspace{0.3cm}

\noindent {\it Case 2.}
Assume that we have a reduced expression
\begin{equation}\label{Eqn=BetaExpansion}
\beta = \beta_1 \ldots \beta_l =  \overline{\alpha}_k \ldots \overline{\alpha}_{k-i+1} \beta_{i+1} \overline{\gamma}_{j}  \ldots \overline{\gamma}_{1},
\end{equation}
for some $0 \leq i < k, 0 \leq j < m$ with $i+j +1 = l$. Moreover since $\beta \not \in A_{00}$ we assume that $\beta_{i+1} \not \in A^s_{00}(\alpha_{k-i}, \gamma_{j+1}), s=1,2$.

 A representation contained in $\alpha \otimes \beta \otimes \gamma$ can have two different forms that determine Case 2 and Case 3. In Case 2 we assume that  $\alpha_{k-i}, \gamma_{j+1}$ and $\beta_{i+1}$ are representations of the same quantum group. Moreover, we assume that we have a subrepresentation of $\alpha \otimes \beta \otimes \gamma$ of the form $\alpha_1 \ldots  \alpha_{k-i-1} \beta_{i+1}'' \gamma_{j+2} \ldots \gamma_{m}$ where $\beta_{i+1}'' \subseteq \alpha_{k-i} \otimes \beta_{i+1} \otimes \gamma_{j+1}$ is irreducible. $\beta_{i+1}''$ is further non-trivial
 since $\beta \not \in A_{00}$. So the expression $\alpha_1 \ldots  \alpha_{k-i-1} \beta_{i+1}'' \gamma_{j+2} \ldots \gamma_{m}$ is reduced.
   In this case, using that we already observed that $\beta_{i+1} \not \in A^s_{00}(\alpha_{k-i}, \gamma_{j+1}), s=1,2$, so that the sets below are defined, we have
\[
\begin{split}
L^{\alpha, \gamma}_{\beta,   \alpha_1 \ldots  \alpha_{k-i-1} \beta_{i+1}'' \gamma_{j+2} \ldots \gamma_{m} } = & \left\{  \alpha_1 \ldots  \alpha_{k-i-1} \beta_{i+1}' \beta_{i+2} \ldots \beta_l \mid \beta_{i+1}'  \in L^{\alpha_{k-i}, \gamma_{j+1}}_{\beta_{i+1}, \beta_{i+1}''}  \right\}, \\
R^{\alpha, \gamma}_{\beta,   \alpha_1 \ldots  \alpha_{k-i-1} \beta_{i+1}'' \gamma_{j+2} \ldots \gamma_{m} } = & \left\{
 \beta_1 \ldots \beta_i \beta_{i+1}' \gamma_{j+2} \ldots \gamma_{m}
 \mid \beta_{i+1}'  \in R^{\alpha_{k-i}, \gamma_{j+1}}_{\beta_{i+1}, \beta_{i+1}''}  \right\}.
 \end{split}
\]
Since there is by assumption a bijection $v(\: \cdot; \beta_{i+1}, \beta_{i+1}''): L^{\alpha_{k-i}, \gamma_{j+1}}_{\beta_{i+1}, \beta_{i+1}''} \rightarrow R^{\alpha_{k-i}, \gamma_{j+1}}_{\beta_{i+1}, \beta_{i+1}''}$ we may set
\[
v( \alpha_1 \ldots  \alpha_{k-i-1} \beta_{i+1}' \beta_{i+2} \ldots \beta_l; \beta,   \alpha_1 \ldots  \alpha_{k-i-1} \beta_{i+1}'' \gamma_{j+2} \ldots \gamma_{m} )
= \beta_1 \ldots \beta_i v( \beta_{i+1}'; \beta_{i+1}, \beta_{i+1}'') \gamma_{j+2} \ldots \gamma_{m},
\]
for $\beta_{i+1}'  \in L^{\alpha_{k-i}, \gamma_{j+1}}_{\beta_{i+1}, \beta_{i+1}''}$. By the previous this is then a bijection
\[
v( \: \cdot \: ; \beta,   \alpha_1 \ldots  \alpha_{k-i-1} \beta_{i+1}'' \gamma_{j+2} \ldots \gamma_{m} ): L^{\alpha, \gamma}_{\beta,   \alpha_1 \ldots  \alpha_{k-i-1} \beta_{i+1}'' \gamma_{j+2} \ldots \gamma_{m} } \rightarrow
R^{\alpha, \gamma}_{\beta,   \alpha_1 \ldots  \alpha_{k-i-1} \beta_{i+1}'' \gamma_{j+2} \ldots \gamma_{m} }.
\]

\vspace{0.3cm}

\noindent {\it Case 3.} We still assume that $\beta$ is written as \eqref{Eqn=BetaExpansion} and treat the remaining case.
The other form that a representation contained in  $\alpha \otimes \beta \otimes \gamma$ can have is a {\it reduced} expression $\beta'  \beta''$   with either $\beta'  \subseteq  \alpha \otimes (\overline{\alpha}_k \ldots \overline{\alpha}_{k-i-1} \beta_{i+1})$ and $\beta'' \subseteq  (\overline{\gamma}_{j}  \ldots \overline{\gamma}_{1}) \otimes \gamma$, or $\beta'  \subseteq \alpha \otimes ( \overline{\alpha}_k \ldots \overline{\alpha}_{k-i-1})$ and $\beta'' \subseteq (\beta_{i+1}   \overline{\gamma}_{j}  \ldots \overline{\gamma}_{1}) \otimes \gamma$.
 We treat the first of these cases, the second one can be treated similarly. In fact both cases are rather close to Case 1. In this case we get
\[
\begin{split}
L^{\alpha, \gamma}_{\beta,  \beta'  \beta'' } = & \left\{   \beta'  \overline{\gamma}_{j}  \ldots \overline{\gamma}_{1} \right\},  \qquad
R^{\alpha, \gamma}_{\beta,   \beta' \beta''  } =   \left\{   \overline{\alpha}_k \ldots \overline{\alpha}_{k-i-1} \beta_{i+1} \beta''      \right\}. \\
 \end{split}
\]
Therefore we may set the bijection $L^{\alpha, \gamma}_{\beta,  \beta'  \beta'' }  \rightarrow R^{\alpha, \gamma}_{\beta,  \beta'  \beta'' }$ by
\[
v(\beta'  \overline{\gamma}_{j}  \ldots \overline{\gamma}_{1} ; \beta, \beta' \beta'' ) =   \overline{\alpha}_k \ldots \overline{\alpha}_{k-i-1} \beta_{i+1} \beta''.
\]

\vspace{0.3cm}

\begin{rmk}\label{Rmk=Exhaust}
Note that Cases 1, 2 and 3 exhaust all the cases for $\beta \not \in A_{00}$. Indeed the only other possible form that a $\beta \in \Irr(\bG)$ can have is  $\beta = \beta_1 \ldots \beta_l =  \overline{\alpha}_k \ldots \overline{\alpha}_{k-i}  \overline{\gamma}_{j}  \ldots \overline{\gamma}_{1}$ for suitable $i,j$ but those representations are in $A_{00}$. It should also be noted that if $\beta$ falls in Case 1 then $\beta \not \in A$.
\end{rmk}

\vspace{0.3cm}

In the following proof we need the following notation. Let $V: K_1 \otimes K_2 \rightarrow K_3$ and let $W: H_1 \otimes H_2 \rightarrow H_3$ with $K_i$ and $H_i$ Hilbert spaces. Then
\[
V \boxtimes W: K_1 \otimes H_1 \otimes H_2 \otimes K_2 \rightarrow K_3 \otimes H_3
\]
 is the map that sends $\xi_1 \otimes \eta_1 \otimes \eta_2 \otimes \xi_2$ to $V (\xi_1 \otimes \xi_2) \otimes W (\eta_1 \otimes \eta_2)$. Note that if $H_3 = \mathbb{C}$ then the range space simplifies to  $K_3 \otimes H_3 = K_3$.

\begin{prop}
Properties \eqref{Eqn=Item=One} and \eqref{Eqn=Item=Two} of Definition \ref{Dfn=Almost} hold for the above choices.
\end{prop}
\begin{proof}
We treat the three cases described above separately. In Remark \ref{Rmk=Exhaust} we already noted that for $\beta \in \Irr(\bG)$ as in Case 1 we have $\beta \not \in  A  \cup A_{00}$. So in Case 1 we must prove \eqref{Eqn=AlmostOneVanish} and \eqref{Eqn=AlmostTwoVanish} only.

\vspace{0.3cm}

\noindent {\bf Proof of \eqref{Eqn=AlmostOneVanish} in Case 1.}
Take $\beta \in \Irr(\bG)$ as in Case 1 so that $\beta \not \in A \cup A_{00}$.  We recall from the discussion in Case 1 that any irreducible representation contained in $\alpha \otimes \beta \otimes \gamma$ can be written as
a reduced expression of the form $\beta' (\beta_{i+1} \ldots \beta_{j-1}) \beta'', i < j$ with $\beta' \subseteq \alpha \otimes \beta_{1} \ldots \beta_{i}$  and  $\beta''  \subseteq  \beta_j \ldots \beta_l \otimes \gamma$ irreducible. Further, we have one point sets
\[
L^{\alpha, \gamma}_{\beta,  \beta' (\beta_{i+1} \ldots \beta_{j-1}) \beta''} = \{ \beta' \beta_{i+1} \ldots \beta_l \}, \qquad
R^{\alpha, \gamma}_{\beta,  \beta' (\beta_{i+1} \ldots \beta_{j-1}) \beta''} = \{ \beta_1 \ldots \beta_{j-1} \beta''  \}
\]
 and the $v$-bijection maps the one set to the other.
 We therefore conclude that \eqref{Eqn=AlmostOneVanish} equals
\[
\begin{split}
       & \Delta_\beta - \Delta_{\beta' \beta_{i+1} \ldots \beta_l} - \Delta_{\beta_1 \ldots \beta_{j-1} \beta''} + \Delta_{\beta' (\beta_{i+1} \ldots \beta_{j-1}) \beta''}  \\
 =  & \sum_{r=1}^l \Delta_{\beta_r}   - ( \Delta_{\beta' } + \sum_{r=i+1}^l \Delta_{\beta_r} )
 - ( \Delta_{\beta'' } + \sum_{r=1}^{j-1} \Delta_{\beta_r} ) +
  ( \Delta_{\beta'} + \Delta_{\beta'' } + \sum_{r = i+1}^{j-1} \Delta_{\beta_r} ) = 0.
  \\
\end{split}
\]

\vspace{0.3cm}

\noindent {\bf Proof of \eqref{Eqn=AlmostTwoVanish} in Case 1.} To prove \eqref{Eqn=AlmostTwoVanish} we notice that for a suitable choice of phase factors,
\[
\begin{split}
V^{\alpha, \beta}_{\beta' \beta_{i+1} \ldots \beta_l} = & V^{\alpha, \beta_1 \ldots \beta_i}_{\beta'} \otimes \id_{\beta_{i+1} \ldots \beta_l},
\qquad
V^{ \alpha,  \beta_{1} \ldots \beta_{j-1} \beta''}_{ \beta' \beta_{i+1} \ldots \beta_{j-1} \beta''} =  V^{ \alpha,  \beta_1 \ldots \beta_i}_{\beta'} \otimes  \id_{\beta_{i+1} \ldots \beta_{j-1} \beta''},
\\
 V^{ \beta, \gamma}_{ \beta_{1} \ldots \beta_{j-1} \beta''} = & \id_{\beta_{1} \ldots \beta_{j-1}} \otimes  V^{ \beta_j \ldots \beta_l, \gamma}_{\beta''},
  \qquad
V^{\beta' \beta_{i+1} \ldots \beta_l, \gamma}_{\beta' \beta_{i+1} \ldots \beta_{j-1} \beta''} =
\id_{\beta' \beta_{i+1} \ldots \beta_{j-1}}  \otimes V^{\beta_j \ldots \beta_l, \gamma}_{\beta''}.
\end{split}
\]
By using these identities in the first and last equation we find the following. The second equation is elementary since the intertwiners commute as they act on different tensor legs. So we get
\[
\begin{split}
V^{\beta' \beta_{i+1} \ldots \beta_l, \gamma}_{\beta' \beta_{i+1} \ldots \beta_{j-1} \beta''} \circ (V^{\alpha, \beta}_{\beta' \beta_{i+1} \ldots \beta_l} \otimes \id_\gamma)
=&
(\id_{\beta' \beta_{i+1} \ldots \beta_{j-1}}  \otimes V^{\beta_j \ldots \beta_l, \gamma}_{\beta''}) \circ (V^{\alpha, \beta_1 \ldots \beta_i}_{\beta'} \otimes \id_{\beta_{i+1} \ldots \beta_l} \otimes \id_\gamma)\\
= &
 (V^{\alpha, \beta_1 \ldots \beta_i}_{\beta'} \otimes \id_{\beta_{i+1} \ldots \beta_{j-1} \beta''} ) \circ (\id_\alpha \otimes \id_{  \beta_{1} \ldots \beta_{j-1}}  \otimes V^{\beta_j \ldots \beta_l, \gamma}_{\beta''}) \\
 = &
V^{ \alpha,  \beta_{1} \ldots \beta_{j-1} \beta''}_{ \beta' \beta_{i+1} \ldots \beta_{j-1} \beta''} \circ (\id_\alpha \otimes V^{ \beta, \gamma}_{ \beta_{1} \ldots \beta_{j-1} \beta''}).
\end{split}
\]
This proves that \eqref{Eqn=AlmostTwoVanish} is true for $\beta$ as in Case 1.

\vspace{0.3cm}

\noindent {\bf Proof of \eqref{Eqn=AlmostOne} and \eqref{Eqn=AlmostOneVanish} in Case 2.}
 Now let $\beta \in \Irr(\bG)$ and assume that we are in Case 2. So $ \beta \not \in   A_{00}$.  Take $\beta'' \subseteq \alpha \otimes \beta \otimes \gamma$ which in Case 2 is assumed to be of the form of a reduced expression $\alpha_1 \ldots  \alpha_{k-i-1} \beta_{i+1}'' \gamma_{j+2} \ldots \gamma_{m}$ where $\alpha_{k-i}, \beta_{i+1}, \gamma_{j+1}$ are representations of the same quantum group and
 $\beta_{i+1}'' \subseteq \alpha_{k-i} \otimes \beta_{i+1} \otimes \gamma_{j+1}$ is irreducible, non-trivial and not contained in $A^s_{00}(\alpha_{k-i}, \gamma_{j+1})$. Take $\beta_{i+1}' \in L^{\alpha_{k-i}, \gamma_{j+1}}_{\beta_{i+1}, \beta_{i+1}''}$
 so that
 \[
 \beta' := \alpha_1 \ldots  \alpha_{k-i-1} \beta_{i+1}' \beta_{i+2} \ldots \beta_l  \in L^{\alpha, \gamma}_{\beta,   \alpha_1 \ldots  \alpha_{k-i-1} \beta_{i+1}'' \gamma_{j+2} \ldots \gamma_{m} }
  \]
  and the $v$-image of $\beta'$ is $\beta_1 \ldots \beta_{i} v(\beta_{i+1}') \gamma_{j+2} \ldots \gamma_{m}$. We have that
 \[
 \begin{split}
  \Delta_{\beta}
 - \Delta_{\beta'}
 - \Delta_{v(\beta')}
 + \Delta_{\beta''}
 = & (  \sum_{r=1}^{l} \Delta_{\beta_r}   )
 - (  \sum_{r=1}^{k-i-1} \Delta_{\alpha_r} + \Delta_{\beta_{i+1}'} + \sum_{r=i+2}^{l} \Delta_{\beta_r} ) \\
& \qquad - (  \sum_{r=1}^{i} \Delta_{\beta_r} + \Delta_{v(\beta_{i+1}')} + \sum_{r=j+2}^{m} \Delta_{\gamma_r} )
 + (  \sum_{r=1}^{k-i-1} \Delta_{\alpha_r} + \Delta_{\beta_{i+1}''} + \sum_{r=j+2}^{m} \Delta_{\gamma_r}  ) \\
 = & \Delta_{\beta_{i+1}}
 - \Delta_{\beta_{i+1}'}
 - \Delta_{v(\beta_{i+1}')}
 + \Delta_{\beta_{i+1}''}.
 \end{split}
 \]
 So since the QMS's on $\bG_1$ and $\bG_2$ are approximately linear we can conclude as follows. In case $\beta_{i+1} \in A^s(\alpha_{k-i}, \gamma_{j+1}), s=1,2$  we see that there is a constant $C>0$ only depending on $\alpha_{k-i}$ and $\gamma_{j+1}$ such that
 \[
 \begin{split}
  \vert \Delta_{\beta_{i+1}}
 - \Delta_{\beta'_{i+1}}
 - \Delta_{v(\beta'_{i+1})}
 + \Delta_{\beta''_{i+1}} \vert \leq & C  \qdim(\beta_{i+1})^{-1}
 \end{split}
 \]
 So by \eqref{Eqn=BetaExpansion} and multiplicativity of the quantum dimension,
 \[
 \begin{split}
  \vert \Delta_{\beta_{i+1}}
 - \Delta_{\beta'_{i+1}}
 - \Delta_{v(\beta'_{i+1})}
 + \Delta_{\beta''_{i+1}} \vert  \leq &
 C  ( \prod_{\substack{ r=1 \ }}^{i} \qdim(\alpha_r) )   ( \prod_{\substack{ r=i+2 \ }}^{m} \qdim(\gamma_r) )  \qdim(\beta)^{-1}.
 \end{split}
 \]
This concludes \eqref{Eqn=AlmostOne}.
 In case  $\beta_{i+1} \not \in A^s(\alpha_{k-i}, \gamma_{j+1}), s=1,2$, and as we assumed that also   $\beta_{i+1} \not \in A^s_{00}(\alpha_{k-i}, \gamma_{j+1}), s=1,2$, we find
 \[
 \vert \Delta_{\beta_{i+1}}
 - \Delta_{\beta'_{i+1}}
 - \Delta_{v(\beta'_{i+1})}
 + \Delta_{\beta''_{i+1}} \vert = 0,
 \]
 and we conclude \eqref{Eqn=AlmostOneVanish}.

\vspace{0.3cm}

\noindent {\bf Proof of \eqref{Eqn=AlmostOneB}   in Case 2.}
We stay in the setting of the previous subproof and assume that $\beta_{i+1} \in A^s(\alpha_{k-i}, \gamma_{j+1}), s=1,2$.  Recall that in Case 2 we have that $\beta'$ must be of the form $\alpha_1 \ldots \alpha_{k-i-1} \beta'_{i+1} \beta_{i+2} \ldots \beta_l$ with $1 \not = \beta'_{i+1} \in L^{\alpha_{k-i}, \gamma_{j+1}}_{\beta_{i+1},  \beta_{i+1}'' }$. In that case $\beta_1 \ldots \beta_{i} = \overline{\alpha}_k \ldots \overline{\alpha}_{k-i+1}$. This gives
\[
\begin{split}
\Delta_\beta - \Delta_{\beta'}     = &    (  \sum_{r=1}^{l} \Delta_{\beta_r}  )  -  ( \sum_{r=i+2}^{l} \Delta_{\beta_r} +  \Delta_{\beta_{i+1}' } + \sum_{r=1}^{k-i+1}  \Delta_{\alpha_r}  )  \\
    = &    \sum_{r=0}^{i-1} \Delta_{\overline{\alpha}_{k-r}}   +  \Delta_{\beta_{i+1}}   -       \Delta_{\beta_{i+1}'} - \sum_{r=1}^{k-i+1}  \Delta_{\alpha_r}.
 \end{split}
\]
We therefore estimate
\[
\vert \Delta_\beta - \Delta_{\beta'} \vert \leq
\vert \sum_{r=0}^{i-1} \Delta_{\overline{\alpha}_{k-r}}    - \sum_{r=1}^{k-i+1}  \Delta_{\alpha_r} \vert +  \vert \Delta_{\beta_{i+1}}   -       \Delta_{\beta_{i+1}'} \vert \leq
\vert \sum_{r=0}^{i-1} \Delta_{\overline{\alpha}_{k-r}}    - \sum_{r=1}^{k-i+1}  \Delta_{\alpha_r} \vert  + C,
\]
for some  constant $C$ that only depends on $\alpha$ and $\gamma$ since both the QMS's on $\bG_1$ and $\bG_2$ are approximately linear.
This proves that \eqref{Eqn=AlmostOneB}  holds for $\beta \in A$ as in Case 2.

\vspace{0.3cm}

\noindent {\bf Proof of \eqref{Eqn=AlmostTwo} and \eqref{Eqn=AlmostTwoVanish} in Case 2.}
  To prove \eqref{Eqn=AlmostTwo} for Case 2 note that up to a phase factor
\[
\begin{split}
V^{\alpha, \beta}_{ \alpha_1 \ldots  \alpha_{k-i-1} \beta_{i+1}' \beta_{i+2} \ldots \beta_l } = & V^{\alpha, \beta_1 \ldots \beta_{i+1}}_{\alpha_1 \ldots  \alpha_{k-i-1} \beta_{i+1}'} \otimes \id_{\beta_{i+2} \ldots \beta_l}, \\
V^{ \alpha,  \beta_1 \ldots \beta_{i} v(\beta_{i+1}') \gamma_{j+2} \ldots \gamma_{m}}_{ \alpha_1 \ldots  \alpha_{k-i-1} \beta_{i+1}'' \gamma_{j+2} \ldots \gamma_{m} } = & V^{ \alpha,  \beta_1 \ldots \beta_{i} v(\beta_{i}')}_{  \alpha_1 \ldots  \alpha_{k-i-1} \beta_{i+1}''  }  \otimes  \id_{ \gamma_{j+2} \ldots \gamma_{m} },
\\
 V^{ \beta, \gamma}_{\beta_1 \ldots \beta_{i} v(\beta_{i+1}') \gamma_{j+2} \ldots \gamma_{m} } = & \id_{\beta_{1} \ldots \beta_{i}} \otimes  V^{ \beta_{i+1} \ldots \beta_l, \gamma}_{ v(\beta_{i+1}') \gamma_{j+2} \ldots \gamma_{m}},\\
V^{\alpha_1 \ldots  \alpha_{k-i-1} \beta_{i+1}' \beta_{i+2} \ldots \beta_l , \gamma}_{ \alpha_1 \ldots  \alpha_{k-i-1} \beta_{i+1}'' \gamma_{j+2} \ldots \gamma_{m}  } = &
\id_{\alpha_1 \ldots  \alpha_{k-i-1} }  \otimes V^{\beta_{i+1}' \beta_{i+2} \ldots \beta_l , \gamma}_{   \beta_{i+1}'' \gamma_{j+2} \ldots \gamma_{m}  }.
\end{split}
\]
Write $x \approx_D y$ for $\Vert x - y \Vert \leq D$. Let $D = C  \qdim(\beta_{i+1})^{-1}$ if $\beta_{i+1}'  \in A^s( \alpha_{k-i}  ,  \gamma_{j+1}  )$ and let $D=0$ otherwise.
We find since $\bG_1$ and $\bG_2$ have almost commuting intertwiners that
\[
\begin{split}
&   V^{\alpha_1 \ldots  \alpha_{k-i-1} \beta_{i+1}' \beta_{i+2} \ldots \beta_l , \gamma}_{ \alpha_1 \ldots  \alpha_{k-i-1} \beta_{i+1}'' \gamma_{j+2} \ldots \gamma_{m}  }  \circ ( V^{\alpha, \beta}_{ \alpha_1 \ldots  \alpha_{k-i-1} \beta_{i+1}' \beta_{i+2} \ldots \beta_l } \otimes  \id_\gamma )\\
 = &
(\id_{\alpha_1 \ldots  \alpha_{k-i-1} }  \otimes V^{\beta_{i+1}' \beta_{i+2} \ldots \beta_l , \gamma}_{   \beta_{i+1}'' \gamma_{j+2} \ldots \gamma_{m}  } )
\circ
(V^{\alpha, \beta_1 \ldots \beta_{i+1}}_{\alpha_1 \ldots  \alpha_{k-i-1} \beta_{i+1}'} \otimes \id_{\beta_{i+2} \ldots \beta_l} \otimes \id_\gamma) \\
= &
 (\id_{\alpha_1 \ldots  \alpha_{k-i-1} }  \otimes V^{\beta_{i+1}'  , \gamma_{j+1}}_{   \beta_{i+1}''   }  \boxtimes V_1^{\beta_{i+2} \ldots \beta_l, \gamma_1 \ldots \gamma_{j}}
 \otimes \id_{\gamma_{j+2} \ldots \gamma_m}
  )\\
  & \qquad
\circ
(\id_{\alpha_1 \ldots  \alpha_{k-i -1} }  \otimes  V^{\alpha_{k-i}, \beta_{i+1}}_{   \beta_{i+1}'  }  \boxtimes V^{  \alpha_{k-i+1} \ldots \alpha_k , \beta_1 \ldots \beta_{i} }_1
  \otimes \id_{\beta_{i+2} \ldots \beta_l} \otimes \id_\gamma) \\
 \approx_{ D }
&
(\id_{\alpha_1 \ldots  \alpha_{k-i -1} } \otimes V^{\alpha_{k-i}, v(\beta_{i+1}')}_{   \beta_{i+1}''  }    \boxtimes V^{  \alpha_{k-i+1} \ldots \alpha_k , \beta_1 \ldots \beta_{i} }_1
   \otimes \id_{\gamma_{j+2} \ldots \gamma_m}  )
\\
& \qquad \circ
 (\id_{\alpha } \otimes \id_{\beta_1 \ldots \beta_{i}} \otimes V^{\beta_{i+1}  , \gamma_{j+1}}_{  v( \beta_{i+1}' )  }  \boxtimes V_1^{\beta_{i+2} \ldots \beta_l, \gamma_1 \ldots \gamma_{j}} \otimes \id_{\gamma_{j+2} \ldots \gamma_m} )
 \\
  = &
(V^{ \alpha,  \beta_1 \ldots \beta_{i} v(\beta_{i+1}')}_{  \alpha_1 \ldots  \alpha_{k-i-1} \beta_{i+1}''  }  \otimes  \id_{ \gamma_{j+2} \ldots \gamma_{m} })
\circ
 (\id_\alpha \otimes \id_{\beta_{1} \ldots \beta_{i}} \otimes  V^{ \beta_{i+1} \ldots \beta_l, \gamma}_{ v(\beta_{i+1}') \gamma_{j+2} \ldots \gamma_{m}})
\\
= &  V^{ \alpha,  \beta_1 \ldots \beta_{i} v(\beta_{i+1}') \gamma_{j+2} \ldots \gamma_{m}}_{ \alpha_1 \ldots  \alpha_{k-i-1} \beta_{i+1}'' \gamma_{j+2} \ldots \gamma_{m} }  \circ (\id_{\alpha }   \otimes  V^{ \beta, \gamma}_{\beta_1 \ldots \beta_{i} v(\beta_{i+1}') \gamma_{j+2} \ldots \gamma_{m} }).
\end{split}
\]
 So that \eqref{Eqn=AlmostTwo} and \eqref{Eqn=AlmostTwoVanish} hold in Case 2.

\vspace{0.3cm}

\noindent {\bf Proof of \eqref{Eqn=AlmostOne} and \eqref{Eqn=AlmostOneVanish} in Case 3.} We turn to Case 3. We shall write
\[
\beta = \beta_1 \ldots \beta_l =  \overline{\alpha}_k \ldots \overline{\alpha}_{k-i+1} \beta_{i+1} \overline{\gamma}_{j}  \ldots \overline{\gamma}_{1}.
\]
This case is essentially the same as Case 1 for $i+1 = j$ (so that the terms $\beta_{i+1} \ldots \beta_{j-1}$ in the proof of Case 1 vanish). Nevertheless we provide full details here.

Consider the subrepresentation of $\alpha \otimes \beta \otimes \gamma$ given by the reduced word $\beta' \beta''$ where $\beta'  \subseteq  \alpha  \otimes (\overline{\alpha}_k \ldots \overline{\alpha}_{k-i+1} \beta_{i+1})$ and $\beta'' \subseteq  (\overline{\gamma}_{j}  \ldots \overline{\gamma}_{1}) \otimes \gamma$ (the case where $\beta'  \subseteq \alpha \otimes ( \overline{\alpha}_k \ldots \overline{\alpha}_{k-i+1})$ and $\beta'' \subseteq (\beta_{i+1}   \overline{\gamma}_{j}  \ldots \overline{\gamma}_{1}) \otimes \gamma$ can be treated in the same manner, or by taking adjoints). We recall that
\[
\begin{split}
L^{\alpha, \gamma}_{\beta,  \beta'  \beta'' } = & \left\{   \beta'  \overline{\gamma}_{j}  \ldots \overline{\gamma}_{1} \right\},  \qquad
R^{\alpha, \gamma}_{\beta,   \beta' \beta''  } =   \left\{   \overline{\alpha}_k \ldots \overline{\alpha}_{k-i+1} \beta_{i+1} \beta''      \right\}. \\
 \end{split}
\]
We write short-hand $v$ for the bijection between these singleton sets. We now find that
\[
\begin{split}
& \Delta_{\overline{\alpha}_k \ldots \overline{\alpha}_{k-i+1} \beta_{i+1}  \overline{\gamma}_{j}  \ldots \overline{\gamma}_{1}   }
- \Delta_{ \beta'  \overline{\gamma}_{j}  \ldots \overline{\gamma}_{1}}
-  \Delta_{\overline{\alpha}_k \ldots \overline{\alpha}_{k-i+1} \beta_{i+1} \beta'' }
+ \Delta_{\beta' \beta''} \\
= & (\Delta_{\beta_{i+1}} +    \sum_{r= k-i+1}^k \Delta_{\overline{\alpha}_{r}}   + \sum_{r=1}^j \Delta_{\overline{\gamma}_r}  )
 - (  \Delta_{\beta'}    +    \sum_{r=1}^j \Delta_{\overline{\gamma}_r}  ) \\
&- (  \Delta_{\beta_{i+1}} + \Delta_{\beta''}   +  \sum_{r= k-i+1}^k \Delta_{\overline{\alpha}_{r}}  )
  +  ( \Delta_{\beta'} + \Delta_{\beta'' }  ) = 0.
\end{split}
\]
This  proves \eqref{Eqn=AlmostOneVanish} and certainly \eqref{Eqn=AlmostOne}; in fact the expression  always is 0.

\vspace{0.3cm}

\noindent {\bf Proof of \eqref{Eqn=AlmostOneB} in Case 3.} \eqref{Eqn=AlmostOneB} can be proved as in Case 2   and we omit the details here.

\vspace{0.3cm}

\noindent {\bf Proof of \eqref{Eqn=AlmostTwo} and \eqref{Eqn=AlmostTwoVanish} in Case 3.} For suitable phase factors for the intertwiners we have
\[
\begin{split}
V^{\alpha, \beta}_{\beta' \beta_{i+2} \ldots \beta_l} = & V^{\alpha, \beta_1 \ldots \beta_{i+1}}_{\beta'} \otimes \id_{\beta_{i+2} \ldots \beta_l},
\qquad
V^{ \alpha,  \beta_{1} \ldots \beta_{i+1} \beta''}_{ \beta'   \beta''} =  V^{ \alpha,  \beta_1 \ldots \beta_{i+1}}_{\beta'} \otimes  \id_{  \beta''},
\\
 V^{ \beta, \gamma}_{ \beta_{1} \ldots \beta_{i+1} \beta''} = & \id_{\beta_{1} \ldots \beta_{i+1}} \otimes  V^{ \beta_{i+2} \ldots \beta_l, \gamma}_{\beta''},
  \qquad
V^{\beta' \beta_{i+2} \ldots \beta_l, \gamma}_{\beta'  \beta''} =
\id_{\beta'  }  \otimes V^{\beta_{i+2} \ldots \beta_l, \gamma}_{\beta''}.
\end{split}
\]
By using these identities in the first and last equation we find the following. The second equation is elementary since the intertwiners commute as they act on different tensor legs. So we get
\[
\begin{split}
V^{\beta' \beta_{i+2} \ldots \beta_l, \gamma}_{\beta' \beta''} \circ (V^{\alpha, \beta}_{\beta' \beta_{i+2} \ldots \beta_l} \otimes \id_\gamma)
=&
(\id_{\beta'  }  \otimes V^{\beta_{i+2} \ldots \beta_l, \gamma}_{\beta''}) \circ (V^{\alpha, \beta_1 \ldots \beta_{i+1}}_{\beta'} \otimes \id_{\beta_{i+2} \ldots \beta_l} \otimes \id_\gamma)\\
= &
 (V^{\alpha, \beta_1 \ldots \beta_{i+1}}_{\beta'} \otimes \id_{  \beta''} ) \circ (\id_\alpha \otimes \id_{  \beta_{1} \ldots \beta_{i+1}}  \otimes V^{\beta_{i+2} \ldots \beta_l, \gamma}_{\beta''}) \\
 = &
V^{ \alpha,  \beta_{1} \ldots \beta_{i+1} \beta''}_{ \beta'  \beta''} \circ (\id_\alpha \otimes V^{ \beta, \gamma}_{ \beta_{1} \ldots \beta_{i+1} \beta''}).
\end{split}
\]
This proves that  \eqref{Eqn=AlmostTwoVanish} and certainly \eqref{Eqn=AlmostTwo} are true in Case 3.
\end{proof}

\section{Approximate linearity with almost commuting intertwiners implies immediately gradient-$\cS_2$}\label{Sect=Solid}
One of the main tools introduced in \cite{CaspersGradient} is the notion for a QMS being immediately gradient Hilbert-Schmidt or immediately gradient-$\cS_2$ where $\cS_2$ refers to the Schatten-von Neumann non-commutative $L_2$-space. The aim of this section is to show that if a QMS is approximately linear with almost commuting intertwiners then it is immediately gradient-$\cS_2$. The immediately gradient-$\cS_2$ property together with some additional assumptions imply rigidity results von Neumann algebras.  The proofs of the latter facts were given in \cite{CaspersGradient} and \cite{CIW} and shall not be repeated here.

We note in the following definition that since $\Phi$ is a QMS of central multipliers that the $\ast$-algebra $\Pol(\bG)$ is in the domain of the generator $\Delta$.

\begin{dfn}
Let $\Phi = (\exp(-t \Delta))_{t \geq 0}$ be a QMS of central multipliers on a compact quantum group $\bG$.    $\Phi$ is called immediately gradient-$\cS_2$ if for every $a, c \in \Pol(\bG)$ the map
\[
\Psi^{a,c}_t: x \Omega_\varphi \mapsto \exp(-t\Delta)(\Delta(axc)  - \Delta(ax)c - a \Delta(xc) + a \Delta(x) c) \Omega_\varphi, \qquad x \in \Pol(\bG),
\]
is bounded $L_2(\bG) \rightarrow L_2(\bG)$ for $t \geq 0$ and moreover Hilbert-Schmidt for $t >0$.
\end{dfn}

We first need the following estimate for the isotypical projections.

\begin{prop}\label{Prop=AuxilliaryEstimate}
Suppose that $\Phi$   is approximately linear with almost commuting intertwiners. Let $a, c \in \Pol(\bG)$ be matrix coeffiecients of respectively $\alpha, \gamma \in \Irr(\bG)$. Let $A_{00} = A_{00}(\alpha, \gamma)$ and the $v$-map  be as in Definition \ref{Dfn=Almost}.  There exists a constant $C = C(a,c) >0$ such that for all $\beta \in \Irr(\bG) \backslash A_{00}, (\beta_1, \beta_2) \in L^{\alpha, \gamma}_\beta$ and every matrix coefficient $x$ of $\beta$ we have
 \[
 \Vert P_{\beta_2} (   P_{\beta_1}(a x) c) - P_{\beta_2} ( a P_{v(\beta_1; \beta, \beta_2)}( xc ))  \Vert_2 \leq   C \qdim(\beta)^{-1} \delta(\beta \in A)  \Vert x \Vert_2,
 \]
 where $v$ and $A$ are as in definition \ref{Dfn=Almost}.
\end{prop}
 \begin{proof}

In this proof we identify $L_\infty(\bG) \otimes M_s(\mathbb{C})$ with $M_s(L_\infty(\bG))$. For an element $X \in M_s(L_\infty(\bG))$ and vectors $\xi, \eta \in \mathbb{C}^s$ we thus have under this correspondence  $\langle X \xi, \eta \rangle = (\id \otimes \omega_{\xi, \eta})(X) \in L_\infty(\bG)$. We shall also write $m := (0, \ldots, 0, 1, 0, \ldots, 0)^t, 1 \leq m \leq s$ (1 at the $m$-th coordinate) for the orthonormal basis vectors in $\mathbb{C}^s$. Let $u^\alpha, u^\beta$ and $u^\gamma$ be some concrete representatives for $\alpha, \beta$ and $\gamma$.

Set $a = \langle u^\alpha i, j \rangle, c = \langle u^{\gamma} m, n \rangle$ and $x = \langle u^\beta k, l \rangle$ with $\beta \in \Irr(\bG) \backslash A_{00}$. By the Woronowicz quantum Peter-Weyl theorem \eqref{Eqn=PeterWeyl} we find
\begin{equation}\label{Eqn=PeterWeyl}
\Vert a \Vert_2^2 =  \qdim(\alpha)^{-1} \langle Q_\alpha^{\frac{1}{2}} i,  Q_\alpha^{\frac{1}{2}}  i \rangle, \quad \Vert x \Vert_2^2 = \qdim(\beta)^{-1} \langle Q_\beta^{\frac{1}{2}} k,  Q_\beta^{\frac{1}{2}}  k \rangle, \qquad  \Vert c \Vert_2^2 = \qdim(\gamma)^{-1} \langle Q_\gamma^{\frac{1}{2}} m,  Q_\gamma^{\frac{1}{2}} m \rangle.
\end{equation}
 We have
\[
  P_{\beta_2}( a P_{v(\beta_1;\beta_2)}(x c)) = \langle u^{ \beta_2 }  (V_{ \beta_2 }^{\alpha,  v(\beta_1; \beta_2) })^\ast  (1 \otimes V_{v(\beta_1; \beta_2)}^{ \beta, \gamma} )^\ast   i \otimes m \otimes k ,  (V_{ \beta_2 }^{\alpha,  v(\beta_1; \beta_2) })^\ast  (1 \otimes V_{v(\beta_1; \beta_2)}^{ \beta, \gamma} )^\ast j \otimes n \otimes l \rangle,
\]
and
\[
P_{\beta_2}(  P_{\beta_1}(a x) c) = \langle u^{\beta_2}( V_{ \beta_2 }^{  \beta_1 \otimes \gamma }  )^\ast   (  V_{\beta_1}^{ \alpha, \beta} \otimes 1 )^\ast   i \otimes m \otimes k , ( V_{ \beta_2 }^{  \beta_1 \otimes \gamma }  )^\ast   (  V_{\beta_1}^{ \alpha, \beta} \otimes 1 )^\ast j \otimes n \otimes l \rangle.
\]
For any $z \in \mathbb{T}$ we have
\begin{equation}\label{Eqn=BigEqn}
\begin{split}
 &  P_{\beta_2}( a P_{v(\beta_1;\beta_2)}(x c))  - P_{\beta_2}(  P_{\beta_1}(a x) c)  \\
=  &
\langle u^{\beta_2 } (( V_{ \beta_2 }^{\alpha,  v(\beta_1; \beta_2) })^\ast  (1 \otimes V_{v(\beta_1; \beta_2)}^{ \beta, \gamma} )^\ast
 -
 z  ( V_{ \beta_2 }^{  \beta_1 \otimes \gamma }  )^\ast   (  V_{\beta_1}^{ \alpha, \beta} \otimes 1 )^\ast)  i \otimes m \otimes k ,  ( V_{ \beta_2 }^{\alpha,  v(\beta_1; \beta_2) })^\ast  (1 \otimes V_{v(\beta_1; \beta_2)}^{ \beta, \gamma} )^\ast  j \otimes n \otimes l \rangle \\
& - \langle u^{  \beta_2  }  z   ( V_{ \beta_2 }^{  \beta_1 \otimes \gamma }  )^\ast   (  V_{\beta_1}^{ \alpha, \beta} \otimes 1 )^\ast   i \otimes m \otimes k,  (  z   ( V_{ \beta_2 }^{  \beta_1 \otimes \gamma }  )^\ast   (  V_{\beta_1}^{ \alpha, \beta} \otimes 1 )^\ast  - ( V_{ \beta_2 }^{\alpha,  v(\beta_1; \beta_2) })^\ast  (1 \otimes V_{v(\beta_1; \beta_2)}^{ \beta, \gamma} )^\ast ) j \otimes n \otimes l \rangle.
\end{split}
\end{equation}
We shall estimate the last two lines. The norm of the first of these lines can be expressed by Peter-Weyl \eqref{Eqn=PeterWeyl} as
\[
\begin{split}
&    \Vert \langle u^{\beta_2 } (( V_{ \beta_2 }^{\alpha,  v(\beta_1; \beta_2) })^\ast  (1 \otimes V_{v(\beta_1; \beta_2)}^{ \beta, \gamma} )^\ast
  -
 z  ( V_{ \beta_2 }^{  \beta_1 \otimes \gamma }  )^\ast   (  V_{\beta_1}^{ \alpha, \beta} \otimes 1 )^\ast ) i \otimes m \otimes k ,  ( V_{ \beta_2 }^{\alpha,  v(\beta_1; \beta_2) })^\ast  (1 \otimes V_{v(\beta_1; \beta_2)}^{ \beta, \gamma} )^\ast  j \otimes n \otimes l \rangle
 \Vert_2   \\
=  & \qdim(\beta_2)^{-1} \Vert Q_{\beta_2}^{\frac{1}{2}} (( V_{ \beta_2 }^{\alpha,  v(\beta_1; \beta_2) })^\ast  (1 \otimes V_{v(\beta_1; \beta_2)}^{ \beta, \gamma} )^\ast
 -
 z  ( V_{ \beta_2 }^{  \beta_1 \otimes \gamma }  )^\ast   (  V_{\beta_1}^{ \alpha, \beta} \otimes 1 )^\ast ) i \otimes m \otimes k \Vert_2 \\
   & \qquad \qquad  \times \qquad \Vert  ( V_{ \beta_2 }^{\alpha,  v(\beta_1; \beta_2) })^\ast  (1 \otimes V_{v(\beta_1; \beta_2)}^{ \beta, \gamma} )^\ast  j \otimes n \otimes l \Vert_2 \\
\leq  & \qdim(\beta_2)^{-1} \Vert  (( V_{ \beta_2 }^{\alpha,  v(\beta_1; \beta_2) })^\ast  (1 \otimes V_{v(\beta_1; \beta_2)}^{ \beta, \gamma} )^\ast
 -
 z  ( V_{ \beta_2 }^{  \beta_1 \otimes \gamma }  )^\ast   (  V_{\beta_1}^{ \alpha, \beta} \otimes 1 )^\ast )  Q_{\alpha}^{\frac{1}{2}} i \otimes Q_{\beta}^{\frac{1}{2}} m \otimes Q_{\gamma}^{\frac{1}{2}} k \Vert_2.
\end{split}
\]
By Definition \ref{Dfn=Almost}  there exists a constant $C>0$ and $z_0 \in \mathbb{T}$ such that
\begin{equation}\label{Eqn=EstOne}
\begin{split}
      &   \Vert \langle u^{\beta_2 } (( V_{ \beta_2 }^{\alpha,  v(\beta_1; \beta_2) })^\ast  (1 \otimes V_{v(\beta_1; \beta_2)}^{ \beta, \gamma} )^\ast
 -
 z_0  ( V_{ \beta_2 }^{  \beta_1 \otimes \gamma }  )^\ast   (  V_{\beta_1}^{ \alpha, \beta} \otimes 1 )^\ast ) i \otimes m \otimes k ,  ( V_{ \beta_2 }^{\alpha,  v(\beta_1; \beta_2) })^\ast  (1 \otimes V_{v(\beta_1; \beta_2)}^{ \beta, \gamma} )^\ast  j \otimes n \otimes l \rangle
 \Vert_2    \\
\leq  &  C \qdim(\beta)^{-1}  \qdim(\beta_2)^{-1} \delta(\beta \in A)  \Vert Q_\alpha^{\frac{1}{2}} i \otimes Q_\beta^{\frac{1}{2}} m \otimes Q_\gamma^{\frac{1}{2}} k \Vert_2.
\end{split}
\end{equation}
Similarly, the second line in \eqref{Eqn=BigEqn} can be estimated with the same $z_0 \in \mathbb{T}$ as
\begin{equation}\label{Eqn=EstTwo}
\begin{split}
& \Vert  \langle u^{  \beta_2  }  z_0   ( V_{ \beta_2 }^{  \beta_1 \otimes \gamma }  )^\ast   (  V_{\beta_1}^{ \alpha, \beta} \otimes 1 )^\ast   i \otimes m \otimes k,  (  z_0   ( V_{ \beta_2 }^{  \beta_1 \otimes \gamma }  )^\ast   (  V_{\beta_1}^{ \alpha, \beta} \otimes 1 )^\ast  - ( V_{ \beta_2 }^{\alpha,  v(\beta_1; \beta_2) })^\ast  (1 \otimes V_{v(\beta_1; \beta_2)}^{ \beta, \gamma} )^\ast ) j \otimes n \otimes l \rangle \Vert_2 \\
\leq &  C \qdim(\beta)^{-1}  \qdim(\beta_2)^{-1} \delta(\beta \in A) \Vert Q_\alpha^{\frac{1}{2}} i \otimes Q_\beta^{\frac{1}{2}} m \otimes Q_\gamma^{\frac{1}{2}} k \Vert_2.
\end{split}
\end{equation}
Combining \eqref{Eqn=BigEqn}, \eqref{Eqn=EstOne} and \eqref{Eqn=EstTwo} we get
\[
\begin{split}
\Vert  P_{\beta_2}( a P_{v(\beta_1;\beta_2)}(x c))  - P_{\beta_2}(  P_{\beta_1}(a x) c)  \Vert_2 \leq & 2  C \qdim(\beta)^{-1}  \qdim(\beta_2)^{-1} \delta(\beta \in A) \Vert Q_\alpha^{\frac{1}{2}} i \otimes Q_\beta^{\frac{1}{2}} m \otimes Q_\gamma^{\frac{1}{2}} k \Vert_2.
\end{split}
\]
Then using \eqref{Eqn=PeterWeyl},
\[
\begin{split}
\Vert P_{\beta_2}( a P_{v(\beta_1;\beta_2)}(x c))  - P_{\beta_2}(  P_{\beta_1}(a x) c) \Vert_2 \leq   & 2 C \qdim(\beta)^{-1} \delta(\beta \in A) \frac{ \qdim(\alpha) \qdim(\beta) \qdim(\gamma)}{ \qdim(\beta_2)} \Vert a \Vert_2 \Vert x \Vert_2 \Vert c \Vert_2.
\end{split}
\]
This concludes the proof since the fraction $\frac{\qdim(\beta)}{ \qdim(\beta_2)}$ is bounded   for all pairs $\beta, \beta_2$ with $\beta_2 \subseteq \alpha \otimes \beta \otimes \gamma$.
 \end{proof}

 \begin{thm}\label{Thm=Main}
 Suppose that $\Phi$ is approximately linear with almost commuting intertwiners. Then $\Phi$ is immediately gradient-$\cS_2$.
 \end{thm}
 \begin{proof}
 We use the same notation as in the proof of Proposition \ref{Prop=AuxilliaryEstimate}.
 Let $a, c \in \Pol(\bG)$ be a matrix coefficient of respectively  $\alpha, \gamma \in \Irr(\bG)$. Say that $a = \langle \alpha \xi, \eta \rangle$ and $c = \langle \gamma \zeta, \psi \rangle$. Let $e_i^\beta, 1 \leq i \leq n_\beta$ be orthogonal vectors in $H_\beta$ such that $\langle \beta e_i^\beta, e_j^\beta \rangle$ is orthogonal in $L_2(\bG)$ \cite[Proposition 2.1]{DawsPrAMS}.
 We must show that for any $t > 0$,
 \begin{equation}\label{Eqn=MainEstimate}
 \sum_{\beta \in \Irr(\bG)} \sum_{i,j=1}^{n_\beta} \frac{ \Vert \Psi_t^{a,b}( \langle \beta e_i, e_j \rangle ) \Vert_2  }{ \Vert  \langle \beta e_i, e_j \rangle \Vert_2  } < \infty.
 \end{equation}
 Let $x = \langle \beta e_i, e_j \rangle$ for some fixed $\beta \in \Irr(\bG), 1 \leq i \leq n_\beta$.
 We start examining the term
 \[
 \begin{split}
 \Psi^{a,b}_0(x) = &
 \Delta( a x c ) -
  \Delta( ax ) c    -  a \Delta( xc )
 +  a \Delta( x) c \\
= &
  \sum_{(\beta_1, \beta_2) \in L_\beta^{\alpha, \gamma}} \left(
\Delta_{\beta_2} P_{\beta_2} (   P_{\beta_1}(a x) c)  -
  \Delta_{\beta_1}  P_{\beta_2}( P_{\beta_1}( ax ) c)  \right)  \\
  &   +
  \sum_{(\beta_1', \beta_2) \in R_\beta^{\alpha, \gamma}} \left( -
  \Delta_{\beta_1'}
  P_{\beta_2} ( a P_{\beta_1'}( xc ))
 +  \Delta_\beta P_{\beta_2} ( a P_{\beta_1'}( x  c ) ) \right).
 \end{split}
 \]
Now if $\beta \not \in A_{00}$ then  we may write this expression as
\[
 \begin{split}
 \Psi^{a,b}_0(x) = &
 \Delta( a x c ) -
  \Delta( ax ) c    -  a \Delta( xc )
 +  a \Delta( x) c \\
= &
 \sum_{(\beta_1, \beta_2) \in L_\beta^{\alpha, \gamma}}
\bigg( \Delta_{\beta_2} P_{\beta_2} (   P_{\beta_1}(a x) c)  -
  \Delta_{\beta_1}  P_{\beta_2}( P_{\beta_1}( ax ) c)    \\
  &   -
  \Delta_{ v(\beta_1; \beta_2) }
  P_{\beta_2} ( a P_{v(\beta_1; \beta_2)}( xc ))
 +  \Delta_\beta P_{\beta_2} ( a P_{ v(\beta_1; \beta_2) }( x  c ) )  \bigg) \\
= &
\sum_{(\beta_1, \beta_2) \in L_\beta^{\alpha, \gamma}}
(\Delta_{\beta_2}  - \Delta_{\beta_1} -    \Delta_{ v(\beta_1; \beta_2) } +  \Delta_\beta   )   P_{\beta_2} ( a P_{v(\beta_1; \beta_2)}( xc )
    \\
  &   +  (   \Delta_{ \beta_1 } -  \Delta_\beta   )  (  P_{\beta_2} ( a P_{v(\beta_1; \beta_2)}( xc ) - P_{\beta_2} (   P_{\beta_1}(a x) c)  ) \\
 \end{split}
 \]
 Since $\Vert  P_{\beta_2} ( a P_{v(\beta_1; \beta_2)}( xc ) \Vert_2 \leq \Vert a \Vert \Vert c \Vert \Vert x \Vert_2$ we estimate
 \begin{equation}\label{Eqn=OneOfMainEstimates}
 \begin{split}
 \Vert  \Psi_t^{a,b}(x) \Vert_2  \leq   &
\sum_{(\beta_1, \beta_2) \in L_\beta^{\alpha, \gamma}}
\exp(-t \Delta_{\beta_2} ) \vert \Delta_{\beta_2}  - \Delta_{\beta_1} -    \Delta_{ v(\beta_1; \beta_2) } +  \Delta_\beta   \vert  \Vert a \Vert \Vert c \Vert \Vert x \Vert_2
    \\
  &   +  \exp(-t \Delta_{\beta_2}) \vert   \Delta_{\beta_1 } -  \Delta_\beta   \vert   \Vert P_{\beta_2} (   P_{\beta_1}(a x) c) - P_{\beta_2} ( a P_{v(\beta_1; \beta_2)}( xc ))  \Vert_2.
  \end{split}
 \end{equation}
 Since the semi-group is approximately linear with almost commuting intertwiners we see that by Proposition \ref{Prop=AuxilliaryEstimate} that there exists a constant $C >0$ only depending on $a$ and $c$  such that
 \[
 \Vert P_{\beta_2} (   P_{\beta_1}(a x) c) - P_{\beta_2} ( a P_{v(\beta_1; \beta_2)}( xc ))  \Vert_2 \leq C^{\frac{1}{2}} \qdim(\beta)^{-1} \delta(\beta \in A)  \Vert x \Vert_2,
 \]
 as well as
 \[
    \vert \Delta_{\beta_2}  - \Delta_{\beta_1} -    \Delta_{ v(\beta_1; \beta_2) } +  \Delta_\beta   \vert \leq C \qdim(\beta)^{-1}  \delta(\beta \in A),
 \]
 and in case $\beta \in A$,
 \[
 \vert   \Delta_{ \beta_1  } -  \Delta_\beta   \vert  \leq C^{\frac{1}{2}}.
 \]
 Combining this with \eqref{Eqn=OneOfMainEstimates}, and estimating $\exp(-t \Delta_{\beta_2}) \leq C'  \exp(-t \beta)$ for some constant $C' > 0$ for all $\beta,\beta_2$ in the summations, we find
 \[
 \begin{split}
 \Vert  \Psi_t^{a,b}(x) \Vert_2  \leq   & C (1 + \Vert a \Vert \Vert c \Vert)
\qdim(\beta)^{-1}  \delta(\beta \in A)   \sum_{(\beta_1, \beta_2) \in L_\beta^{\alpha, \gamma}}
\exp(-t \Delta_{\beta_2} )      \Vert x \Vert_2 \\
\leq & C C'  (1 + \Vert a \Vert \Vert c \Vert)
\qdim(\beta)^{-1}  \delta(\beta \in A) (\# L^{\alpha, \gamma}_\beta)  \exp(-t \Delta_{\beta} )      \Vert x \Vert_2.
  \end{split}
 \]
By Lemma \ref{Lem=RepBound} we have that $\# L^{\alpha, \gamma}_\beta$ is bounded in $\beta$ with bound depending only on $\alpha$ and $\gamma$. We may therefore assemble terms and conclude that there exists a constant $C(a,c)$ only depending on $a$ and $c$ such that
\begin{equation}\label{Eqn=MidSummary}
 \Vert  \Psi_t^{a,b}(x) \Vert_2 \leq C(a,c)
\qdim(\beta)^{-1}  \delta(\beta \in A)   \exp(-t \Delta_{\beta} )      \Vert x \Vert_2.
\end{equation}
We can now estimate the term \eqref{Eqn=MainEstimate} as follows, where in the last line we use that the classical dimension is smaller than or equal to the quantum dimension,
\begin{equation}\label{Eqn=Main1}
\begin{split}
 \sum_{\beta \in \Irr(\bG) \backslash A_{00}} \sum_{i,j=1}^{n_\beta} \frac{ \Vert \Psi_t^{a,b}( \langle \beta e_i, e_j \rangle ) \Vert_2^2  }{ \Vert  \langle \beta e_i, e_j \rangle \Vert_2^2  }
 \leq & C(a,c)^2 \sum_{\beta \in A } \sum_{i,j=1}^{n_\beta}
\qdim(\beta)^{-2}     \exp(-2 t \Delta_{\beta} ) \\
\leq & C(a,c)^2 \sum_{\beta \in A }
\qdim(\beta)^{-2} n_\beta^2  \exp(-2 t \Delta_{\beta} ) \\
\leq &  C(a,c)^2   \sum_{\beta \in A }    \exp(-2 t \Delta_{\beta} ). \\
\end{split}
\end{equation}
In turn we may estimate using \eqref{Eqn=Item=Three} of Definition \ref{Dfn=Almost} and get
\begin{equation}\label{Eqn=Main2}
\begin{split}
 \sum_{\beta \in A }    \exp(-2 t \Delta_{\beta} )
 = \sum_{N \in \mathbb{N}} \sum_{\substack{ \beta \in A, \\ N < \Delta_\beta \leq N+1}  }    \exp(-2 t N ) \leq
 \sum_{N \in \mathbb{N}} P(N) \exp(-2t N) < \infty.
\end{split}
\end{equation}
Combining \eqref{Eqn=Main1} and \eqref{Eqn=Main2} we see that for $t >0$,
\[
 \sum_{\beta \in \Irr(\bG) \backslash A_{00}} \sum_{i,j=1}^{n_\beta} \frac{ \Vert \Psi_t^{a,b}( \langle \beta e_i, e_j \rangle ) \Vert_2^2  }{ \Vert  \langle \beta e_i, e_j \rangle \Vert_2^2  }
 \leq  \sum_{\beta \in  \Irr(\bG) \backslash A_{00}} \sum_{i,j=1}^{n_\beta} \frac{ \Vert \Psi_t^{a,b}( \langle \beta e_i, e_j \rangle ) \Vert_2^2  }{ \Vert  \langle \beta e_i, e_j \rangle \Vert_2^2  } +
  C(a,c)^2 \sum_{N \in \mathbb{N}} P(N) \exp(-2t N) < \infty.
\]
So \eqref{Eqn=MainEstimate} is finite as $A_{00}$ is finite.

Finally let $x \in \Pol(\bG)$ and write $x_\beta = P_\beta(x)$ so that $x = \sum_{\beta \in \Irr(\bG)} x_\beta$. We have by the triangle inequality, \eqref{Eqn=MidSummary} and the Cauchy-Schwarz inequality  that
\[
\begin{split}
\Vert \Psi^{a,c}_0(x) \Vert_2 \leq  & \Vert \sum_{\beta \in A_{00}}  P_\beta(x) \Vert_2 +   \sum_{\beta \in A} C(a,c ) \qdim(\beta)^{-1} \Vert P_{\beta} x \Vert_2 \\
\leq &  \Vert \sum_{\beta \in A_{00}}  P_\beta(x) \Vert_2+   C(a,c ) ( \sum_{\beta \in A} \qdim(\beta)^{-2} )^{\frac{1}{2}} ( \sum_{\beta \in A} \Vert P_{\beta} x \Vert_2^2 )^{\frac{1}{2}} \\
\leq  &  (1 + C(a,c ) ( \sum_{\beta \in A} \qdim(\beta)^{-2} )^{\frac{1}{2}} )
( \Vert \sum_{\beta \in A_{00}}  P_\beta(x) \Vert_2 +    \Vert  \sum_{\beta \in A}  P_{\beta} x \Vert_2 ) \\
\leq & \sqrt{2} (1 + C(a,c ) ( \sum_{\beta \in A} \qdim(\beta)^{-2} )^{\frac{1}{2}} )
   \Vert  \sum_{\beta \in \Irr(\bG)}  P_{\beta} x \Vert_2.
\end{split}
\]
This gives boundedness of $\Psi^{a,c}_0$ and concludes that $\Phi$ is immediately gradient-$\cS_2$ by \eqref{Eqn=Item=Three} of Definition \ref{Dfn=Almost}.
 \end{proof}

Recall that we say that a QMS is {\it immediately $L_2$-compact} if for every $t >0$ the map $x \Omega_\varphi \mapsto \Phi_t(x) \Omega_\varphi$ is compact as a map on $L_2(\bG)$. Equivalently the generator $\Delta \geq 0$ has compact resolvent.

\begin{thm}\label{Thm=Gradient}
Let $\bG$ be a compact quantum group with the W$^\ast$-CBAP with constant $\Lambda$. Suppose that $\bG$ admits a QMS that is immediately gradient-$\cS_2$ and immediately $L_2$-compact. Then,
\begin{enumerate}
\item\label{Item=StrongSolidKac} If $\bG$ is of Kac type then $L_\infty(\bG)$ is strongly solid.
\item\label{Item=StrongSolidNonKac} If $L_\infty(\bG)$ is solid and $\Lambda =1$  then it is strongly solid.
\end{enumerate}
\end{thm}
\begin{proof}
\eqref{Item=StrongSolidNonKac} was proved in \cite[Proposition 7.9]{CaspersGradient} and is based on the results of \cite{BHV}, \cite{OzawaPopaAJM} and \cite{CiprianiSauvageot}. For  \eqref{Item=StrongSolidKac} we see by \cite[Section 3.2]{CaspersGradient} and \cite[Proposition 3.8]{CaspersGradient} (based on \cite{CiprianiSauvageot}) that there exists a closable real derivation $\partial: \Pol(\bG) \rightarrow H_\partial$ into an $L_\infty(\bG)$-$L_\infty(\bG)$ $H_\partial$ such that $\Delta = \partial^\ast \overline{\partial}$. Further, since $\Phi$ is immediately $L_2$-compact $\Delta$  has compact resolvent. Moreover, by \cite[Proposition 4.3]{CaspersGradient} (see also \cite[Theorem 3.9]{CIW}) this bimodule $H_\partial$ can be constructed in such a way that it is weakly contained in the coarse bimodule  of $L_\infty(\bG)$. It follows then from the main results of \cite[Corollary B]{OzawaPopaAJM} that $L_\infty(\bG)$ is strongly solid; we note that \cite[Corollary B]{OzawaPopaAJM} is only stated for group von Neumann algebras but it holds in this context as well  (see e.g. \cite[Appendix]{CaspersGradient}).
\end{proof}

Combining Theorem \ref{Thm=Gradient} with Theorem \ref{Thm=Main} we conclude the following main results of this paper.

\begin{cor}\label{Cor=KacStronglsySolid}
Let $\bG$ be a compact quantum group of Kac type with the W$^\ast$-CBAP. Suppose that $\bG$ admits a QMS of central multipliers that is approximately linear with almost commuting intertwiners and immediately $L_2$-compact. Then $L_\infty(\bG)$ is strongly solid.
\end{cor}

 We also get the following corollary which shall not be used further in this paper.

\begin{cor}\label{Cor=NonKacStronglsySolid}
Let $\bG$ be a compact quantum group  with the W$^\ast$-CCAP such that $L_\infty(\bG)$ is solid. Suppose that $\bG$ admits a QMS of central multipliers that is approximately linear with almost commuting intertwiners and immediately $L_2$-compact. Then $L_\infty(\bG)$ is strongly solid.
\end{cor}

\section{Quantum Markov semi-groups and differentiable families of states}\label{Sect=QMS}

We prove that $SU_q(2)$ admits a QMS of central multipliers that is approximately linear with almost commuting intertwiners. Parts of the proof compare to our analysis from \cite{CaspersGradient}. However, we present a much more conceptual and shorter approach by making use of generating functionals.  The author is indebted to Adam Skalski for showing him  this argument which is contained in Section \ref{Sect=Adam}.

\subsection{Preliminaries on quantum $SU(2)$}

\begin{dfn}
Let $\bG_q, q \in (-1, 1) \backslash \{ 0 \}$ be the quantum $SU(2)$ group. It may be defined as follows. Consider the Hilbert space $\ell_2(\mathbb{N}_{\geq 0}) \otimes \ell_2(\mathbb{Z})$ with natural orthonormal basis $e_i \otimes f_k, i \in \mathbb{N}_{\geq 0}, k \in \mathbb{Z}$. Define the operators
\[
\begin{split}
\alpha e_i \otimes f_k = & \sqrt{1 - q^{2i}}  e_{i-1}\otimes f_k, \\
\gamma e_i \otimes f_k = & q^i e_i \otimes f_{k+1},
\end{split}
\]
and the comultiplication determined by
\[
\Delta_{\bG_q}(\alpha) = \alpha \otimes \alpha - q \gamma^\ast \otimes \gamma, \qquad \Delta_{\bG_q}(\gamma) = \gamma \otimes \alpha + \alpha^\ast \otimes \gamma.
\]
\end{dfn}

It was proved in \cite{Banica} that  $\Irr(\bG_q) = \mathbb{N}_{\geq 0}$ and the fusion rules of $\bG_q$ are described by
\[
\alpha \otimes \beta = \vert \alpha - \beta \vert \oplus  \vert \alpha - \beta \vert + 2 \oplus \ldots \oplus \vert \alpha + \beta \vert -2 \oplus \vert \alpha + \beta \vert, \qquad \alpha, \beta \in \mathbb{N}_{\geq 0}.
\]

\subsection{QMS's on quantum $SU(2)$}\label{Sect=Adam}
We construct a natural QMS of central multipliers on $\bG_q$, i.e. quantum $SU(2)$. The QMS is the same as the one from \cite[Section 6.1]{CaspersGradient} but the approach is more conceptual.  See also \cite{BrannanCrelle}, \cite{CFY} and \cite{CiprianiKulaFranz} for related results.

\begin{dfn}
A {\bf generating functional} is a (linear) functional $L: \Pol(\bG) \rightarrow \mathbb{C}$ such that $L(1) =0$, such that $L(x^\ast) =  \overline{L(x)}$ (i.e. $L$ is self-adjoint) and such that if for $x \in \Pol(\bG)$ we have $\epsilon(x) = 0$ then $L(x^\ast x ) \leq 0$ (i.e. $L$ is conditionally negative definite).
\end{dfn}

A state on the unital $\ast$-algebra $\Pol(\bG)$ is a map $\mu: \Pol(\bG) \rightarrow \mathbb{C}$ such that $\mu(x^\ast x) \geq 0, x \in \Pol(\bG)$ and $\mu(1) = 1$. Recall that  $\epsilon$ denotes the counit.

\begin{prop}\label{Prop=GeneratingFunctional}
Let $\bG$ be a compact quantum group and let $(\mu_t)_{t \geq 0}$ be a family of states on $\Pol(\bG)$ (not necessarily forming a convolution semi-group). Assume that for every $x \in \Pol(\bG)$ the limit
\[
L(x) := \lim_{t \searrow 0}  \frac{1}{t} (\epsilon(x) - \mu_t(x)),
\]
exists. Then $L: \Pol(\bG) \rightarrow \mathbb{C}$ is a generating functional.
\end{prop}
\begin{proof}
Let $x \in \Pol(\bG)$ be such that $\epsilon(x) = 0$. Then,
\[
\mu_t(x^\ast x) - \epsilon(x^\ast x) =  \mu_t(x^\ast x) - \epsilon(x)^\ast \epsilon(x) = \mu_t(x^\ast x) \geq 0.
\]
All other properties are clear.
\end{proof}

Let $U_\alpha, \alpha \in \mathbb{N}$ be the Chebyshev polynomials of the second kind with derivative $U_\alpha'$. They are orthogonal polynomials satisfying $U_0 = 1, U_1(\lambda) = \lambda$ and the recursion relation
\[
\lambda U_\alpha(\lambda) = U_{\alpha + 1}(\lambda) + U_{\alpha - 1}(\lambda), \qquad \lambda \in \mathbb{R}, \alpha \in \mathbb{N}_{\geq 1}.
\]
In \cite[Theorem 17]{CFY} (see also \cite{BrannanCrelle}, \cite{FimaVergnioux}) it was proved that for every $t \in [-1, 1]$ there exists a state $\mu_t: \Pol(\bG) \rightarrow \mathbb{C}$ characterized by
\begin{equation}\label{Eqn=CFY}
\mu_t   (u^\alpha_{ij}) = \left( \frac{ U_\alpha(q^t + q^{-t}) }{U_\alpha(q + q^{-1})} \right)^{3} \delta_{i,j}, \qquad \alpha \in \mathbb{N}_{\geq 0}, 1 \leq i,j \leq n_\alpha.
\end{equation}

\begin{prop}\label{Prop=ExplicitGamma} There exists a generating functional $L: Pol(\bG_q) \rightarrow \mathbb{C}$ given by
\[
(L\otimes \id) u^\alpha  =  \Delta_\alpha \id_{n_\alpha}, \qquad {\rm with } \qquad \Delta_\alpha = \frac{ U_\alpha'(q^1 + q^{-1}) }{U_\alpha(q + q^{-1})}.
\]
\end{prop}
\begin{proof}
Consider the function,
\[
c_\alpha(t) :=  \left( \frac{ U_\alpha(q^t + q^{-t}) }{U_\alpha(q + q^{-1})} \right)^{3}, \qquad [-1, 1].
\]
The  derivative of this function is,
\[
\begin{split}
c'_\alpha(t )
= & \frac{  U_\alpha'(q^t + q^{-t})}{U_\alpha(q + q^{-1})} (q^t  - q^{-t}  ) \log(q).
\end{split}
\]
Proposition \ref{Prop=GeneratingFunctional} and \eqref{Eqn=CFY} show that there is a generating functional $L_0: \Pol(\bG) \rightarrow \mathbb{C}$ determined by
\[
(L_0 \otimes \id) (u^\alpha)  = c_\alpha'(1)  \id_{n_\alpha}.
\]
Then also $L =  \log(q)^{-1} (q  - q^{-1}   )^{-1} L$ is a generating functional and the proposition is proved.
\end{proof}
\begin{thm}\label{Thm=Markov}
Let $\bG = SU_q(2)$ with $q \in (-1,1) \backslash \{ 0 \}$. There exists a QMS $\Phi = (\Phi_t)_{t \geq 0}$ on $L_\infty(\bG)$ determined by
\[
(\Phi_t \otimes \id) u^\alpha =  \exp(-t \Delta_\alpha)  u^\alpha, \qquad \alpha \in \mathbb{N}_{\geq 0}.
\]
Here $\Delta_\alpha$ is defined in Proposition \ref{Prop=ExplicitGamma}. Moreover, $\Phi$ is approximately linear with almost commuting intertwiners.
\end{thm}
\begin{proof}
Let $L: \Pol(\bG) \rightarrow \mathbb{C}$ be the generating functional from Proposition \ref{Prop=ExplicitGamma}. By \cite[Lemma 6.14]{DFSW} we see that
\[
\exp(-t L) := \sum_{k=0}^\infty \frac{1}{k!} (-t L)^{\ast k},
\]
is a convolution semi-group of states. We set
\[
\Phi_t = (\exp(-t L) \otimes \id) \circ \Delta, \qquad t \geq 0,
\]
which then forms a QMS.  We have, writing $u_{i,j}^{\alpha}$ for the matrix coefficients with respect to some orthonormal basis of $\mathbb{C}^{n_\alpha}$, that
\[
\Phi_t(  u_{ij}^\alpha) =  (\exp(-t L) \otimes \id)\left( \sum_{k=1}^{n_\alpha} u_{ik}^{\alpha} \otimes u_{kj}^{\alpha} \right) = \exp(-t \Delta_\alpha) u_{ij}^{\alpha}.
\]
It follows that $(\Phi_t)_{t \geq 0}$ is a QMS with the desired properties.
\end{proof}

\subsection{Approximate linearity} In the current and next section we prove that the QMS from Theorem \ref{Thm=Markov} is approximately linear with almost commuting intertwiners.
In order to do so we fix the following notation.

Recall that $\Irr(\bG) = \mathbb{N}_{\geq 0}$. Take $\alpha, \gamma \in \mathbb{N}_{\geq 0}$. Set $A_{00} = \{0, 1, \ldots,  \max(\alpha, \gamma)\}$ and let $A = \mathbb{N}_{\geq 0} \backslash A_{00}$. We note that $A$ and $A_{00}$ partition $\mathbb{N}_{\geq 0}$ and therefore we do not need to check \eqref{Eqn=AlmostOneVanish} and \eqref{Eqn=AlmostTwoVanish}.
 Take $\beta \in A$.  Then if $\beta_2 \subseteq \alpha \otimes \beta \otimes \gamma$ we must have $\beta_2 \in \{ \beta - \alpha - \gamma, \beta - \alpha - \gamma + 2, \ldots, \beta + \alpha + \gamma  \}$. We have
\[
\begin{split}
L^{\alpha, \gamma}_{\beta, \beta_2} = & \{ \beta-\alpha, \beta-\alpha+2, \ldots, \beta + \alpha \}, \\
R^{\alpha, \gamma}_{\beta, \beta_2} = & \{ \beta - \gamma, \beta-\gamma+2, \ldots, \beta+\gamma \}.
\end{split}
\]
 We set $v(\beta_1; \beta, \beta_2) = \beta + \beta_2 - \beta_1$.

The proof of the next proposition is the same as \cite[Section 6.1 and 6.2]{CaspersGradient}.

\begin{prop}\label{Prop=ExplicitGamma2}
The QMS defined in Theorem \ref{Thm=Markov} is approximately linear.
\end{prop}
\begin{proof}
For any $m,n \in \mathbb{Z} \backslash \{ 0 \}$ we have
\[
\frac{1 + q^{-2m} }{   1- q^{-2m}}  - \frac{1 + q^{-2n} }{   1- q^{-2n}}
=
\frac{   2 (   q^{-2m} -  q^{-2n} )   }{  ( 1- q^{-2m}) ( 1- q^{-2n})   } = \frac{2 (q^{2n} - q^{2m} )}{  (q^{2m} -1 ) ( q^{2n} -1 )  }.
\]
Let $N_q = q + q^{-1}$ which is the quantum dimension of the fundamental representation.
By \cite[Lemma 4.4]{FimaVergnioux} we have the explicit expression
\begin{equation}\label{Eqn=FimaVergnioux}
\Delta_\alpha = \frac{\alpha}{ \sqrt{ N_q^2 - 4}  } \left(  \frac{1+ q^{-2\alpha -2}}{ 1 - q^{-2 \alpha -2} } \right) + \frac{2}{(1-q^2)  \sqrt{N_q^2 - 4}}.
\end{equation}
Therefore it follows that for $\beta, \beta_1 \in \Irr(\bG)$ we have
\[
\begin{split}
\vert \Delta_\beta - \Delta_{\beta_1} \vert \leq & \vert \beta - \beta_1 \vert   \frac{1}{ \sqrt{ N_q^2 - 4}  }   \frac{1+ q^{-2\beta -2}}{ 1 - q^{-2 \beta -2} }    +
\frac{\beta_1}{ \sqrt{ N_q^2 - 4}    } \left|  \frac{1+ q^{ -2 \beta - 2 }}{ 1 - q^{ -2 \beta - 2 } } -  \frac{1+ q^{-2\beta_1 -2}}{ 1 - q^{-2 \beta_1 -2} }  \right| \\
= &  \vert \beta - \beta_1 \vert   \frac{1}{ \sqrt{ N_q^2 - 4}  }   \frac{1+ q^{-2\beta -2}}{ 1 - q^{-2 \beta -2} }    +
\frac{\beta_1}{ \sqrt{ N_q^2 - 4}    } \left| \frac{2 (q^{2\beta+2} - q^{2\beta_1+2} )}{  (q^{2 \beta +2} -1 ) ( q^{2 \beta_1 +2} -1 )  }  \right|.
\end{split}
\]
This expression can be estimated uniformly over all $\beta, \beta_1 \in \mathbb{N}_{\geq 0}$ with $\vert \beta - \beta_1 \vert \leq \alpha + \gamma$. This yields \eqref{Eqn=AlmostOneB}.  Further,
\[
\begin{split}
   &  \sqrt{N_q^2 - 4} \vert \Delta_{\beta} - \Delta_{\beta_1} - \Delta_{\beta + \beta_2 - \beta_1}   + \Delta_{\beta_2} \vert  \\
=  &   \left|   \beta \frac{1+ q^{-2\beta -2}}{ 1 - q^{-2 \beta -2} } -   \beta_1 \frac{1+ q^{-2\beta_1 -2}}{ 1 - q^{-2 \beta_1 -2} }
-  (\beta + \beta_2 - \beta_1) \frac{1+ q^{-2(\beta + \beta_2 - \beta_1) -2}}{ 1 - q^{-2 (\beta + \beta_2 - \beta_1) -2} }
+  \beta_2 \frac{1+ q^{-2\beta_2 -2}}{ 1 - q^{-2 \beta_2 -2} }
\right| \\
\leq &     \beta \left| \frac{1+ q^{-2\beta -2}}{ 1 - q^{-2 \beta -2} }
-  \frac{1+ q^{-2(\beta + \beta_2 - \beta_1) -2}}{ 1 - q^{-2 (\beta + \beta_2 - \beta_1) -2} } \right|
+ \beta_1 \left|
\frac{1+ q^{-2\beta_1 -2}}{ 1 - q^{-2 \beta_1 -2} } - \frac{1+ q^{-2(\beta + \beta_2 - \beta_1) -2}}{ 1 - q^{-2 (\beta + \beta_2 - \beta_1) -2} }
\right|\\
&
+ \beta_2 \left|
\frac{1+ q^{-2\beta_2 -2}}{ 1 - q^{-2 \beta_2 -2} } - \frac{1+ q^{-2(\beta + \beta_2 - \beta_1) -2}}{ 1 - q^{-2 (\beta + \beta_2 - \beta_1) -2} }
\right| \\
\leq & \beta \vert q^{2 \beta} - q^{2 (\beta + \beta_2 - \beta_1)} \vert + \beta_1 \vert q^{2 \beta_1} - q^{2 (\beta + \beta_2 - \beta_1)} \vert
+ \beta_2 \vert q^{2 \beta_2} - q^{2 (\beta + \beta_2 - \beta_1)} \vert.
\end{split}
\]
As asymptotically $\qdim(\beta) \approx q^{-\beta}$ we see that there exists a constant $C>0$ such that for all  $\beta, \beta_1, \beta_2 \in \mathbb{N}_{\geq 0}$ with $\vert \beta - \beta_1 \vert \leq \alpha + \gamma$ and $\vert \beta - \beta_2 \vert \leq \alpha + \gamma$ we have that
\[
\sqrt{N_q^2 - 4} \vert \Delta_{\beta} - \Delta_{\beta_1} - \Delta_{\beta + \beta_2 - \beta_1}   + \Delta_{\beta_2} \vert  \leq C \beta \qdim(\beta)^{-2} \leq C \qdim(\beta)^{-1}.
\]
This yields the desired estimate \eqref{Eqn=AlmostOne}.
\end{proof}

\subsection{Almost commuting intertwiners}

In this section we extend the results from \cite[Appendix]{VaesVergnioux} on almost commuting intertwiners. In fact these results are self-improving in the sense that the main estimates are already proved in \cite{VaesVergnioux}. Here we show that they automatically imply the same results for a larger range of representations.

\vspace{0.3cm}

The following lemma and proposition pertain to $\bG_q, q \in (-1,1) \backslash \{ 0 \}$. Note however that the principle of proof of Lemma \ref{Lem=InterIso} actually works for any compact quantum group. In the following statements we require that $\alpha +k$ is even or odd (and $\gamma + l$ is even). In other words, that $\alpha$ and $k$ have the same parity or different parity. This is because otherwise the intertwiner $V^{\alpha, \beta}_{\beta+k}$ or  $V^{\alpha+1, \beta}_{\beta+k}$ would  be 0 by the fusion rules and the statements below would thus be trivial.

\begin{lem}\label{Lem=InterIso}
 Let $\alpha, \beta \in \mathbb{N}_{\geq 0}$ with $\alpha \leq \beta$. Let $k  \in \mathbb{Z}$ with $\vert k \vert \leq \alpha$ and $\alpha +k$ odd. Then we have up to a phase factor
\begin{equation}\label{Eqn=IsoInter}
\sum_{k' = - \alpha, -\alpha+2, \ldots, \alpha}  V^{1, \beta + k'}_{\beta + k }  (\id_{1} \otimes   V^{\alpha, \beta}_{\beta + k'} )  (V^{1, \alpha}_{\alpha + 1}  \otimes \id_{\beta})^\ast  = V^{\alpha +1, \beta}_{\beta+k}.
\end{equation}
\end{lem}
\begin{proof}
We may decompose $1 \otimes \alpha \otimes \beta = \oplus_{\delta \in \mathbb{N}_{\geq 0}} m_\delta \cdot \delta$ where $m_\delta$ denotes the multiplicity.
Each of the intertwiners $V^{1, \beta + k'}_{\beta + k }  (\id_{1} \otimes   V^{\alpha, \beta}_{\beta + k'} )$ intertwine $1 \otimes \alpha \otimes \beta$ with a copy of $\beta+k$ and the copies are orthogonal for different $k'$.  Moreover
\[
\sum_{k' = - \alpha, -\alpha+2, \ldots, \alpha}  V^{1, \beta + k'}_{\beta + k }  (\id_{1} \otimes   V^{\alpha, \beta}_{\beta + k'} )
\]
 intertwines $1 \otimes \alpha \otimes \beta$ with $m_{\beta +k} \cdot (\beta+k)$, i.e. it exhaust all the summands.

The  the total expression on the left hand side of \eqref{Eqn=IsoInter} intertwines $(\alpha + 1) \otimes \beta$ with $\beta + k$ and therefore by Schur's lemma must  be a scalar multiple of  the isometry $V^{\alpha +1, \beta+l}_{\beta+k+l}$.
By the first paragraph of this proof and the fact that  $(V^{1, \alpha}_{\alpha + 1}  )^\ast$ is an isometry that maps $\alpha+1$ into its isotypical component in $1 \otimes \alpha$ we find that this scalar multiple must be in $\mathbb{T}$.
\end{proof}

\begin{prop}\label{Prop=AlmostCom}
 Let $\alpha, \gamma \in \mathbb{N}_{\geq 0}$. There exists a constant $C >0$ such that for all $\beta \in \mathbb{N}_{\geq 0}, \beta \geq \max(\alpha, \gamma)$ and $k,l \in \mathbb{Z}$ with $\vert k \vert \leq \alpha, \vert l \vert \leq \gamma$ and $\alpha +k$ and $\gamma +l$ even we have
\begin{equation}\label{Eqn=AlmostCommuteSuq2}
   \inf_{z \in \mathbb{Z}}  \Vert  V^{\beta + k, \gamma}_{\beta + k + l} (V^{\alpha, \beta}_{\beta + k} \otimes \id_\gamma)  - z V^{\alpha, \beta +l}_{\beta + k + l} (\id_{\alpha} \otimes V^{\beta, \gamma}_{\beta + l})  \Vert \leq C \qdim(\beta)^{-1}.
\end{equation}
\end{prop}
\begin{proof}
This lemma was proved in \cite[Lemmas A.1 and A.2]{VaesVergnioux} for the case $\alpha, \gamma = 1$ and $(k,l)$ equal to either $(1,1), (1, -1), (-1, 1)$. For $(k,l) = (-1, -1)$ the same conclusion can be derived as follows. Both
 $V^{\beta -1, \gamma}_{\beta -2} (V^{\alpha, \beta}_{\beta -1} \otimes \id_\gamma)$ and $V^{\alpha, \beta -1}_{\beta -2} (\id_{\alpha} \otimes V^{\beta, \gamma}_{\beta - 1})$ are intertwiners from $\alpha \otimes \beta \otimes \gamma$ to $\beta - 2$. But by the fusion rules $\beta - 2$ occurs at most once in the decomposition of $\alpha \otimes \beta \otimes \gamma$ in terms of irreducibles. Therefore such an intertwiner is unique up to a phase factor.  So the left hand side of \eqref{Eqn=AlmostCommuteSuq2} is 0. We note that actually also in case $(k,l) = (1,1)$ the left hand side of \eqref{Eqn=AlmostCommuteSuq2} is 0 for the analogous reason.

 We now prove the general case by an induction argument. Suppose that the statement is true for $\alpha$ and $\gamma$. Then we shall prove it for $\alpha +1$ and $\gamma$. Consider the composition of maps with $\alpha, \beta, \gamma, k, l$ as in the proposition and $\vert k' \vert \leq \alpha$ such that $\alpha +k'$ is even,
 \[
 \begin{array}{ccccccc}
 A_\beta:  1 \otimes \alpha \otimes \beta \otimes \gamma &\quad\:\: _{\rightarrow}\!\!\!\!\!\!\!\!\!\!\!\!\!\!\!\!\!\!^{  \id_1 \otimes \id_\alpha \otimes V^{\beta, \gamma}_{\beta+l} } &   1 \otimes \alpha \otimes (\beta +l )  & \quad\:\: _{\rightarrow}\!\!\!\!\!\!\!\!\!\!\!\!\!\!\!\!\!\!^{ \id_1 \otimes   V^{\alpha, \beta+l}_{\beta+k'+l}} &   1 \otimes   (\beta + k'  + l ) & \quad\:\: _{\rightarrow}\!\!\!\!\!\!\!\!\!\!\!\!\!^{ V^{1, \beta+k'+l}_{\beta+k+l}}  &  \beta + k  + l, \\
 B_\beta: 1 \otimes \alpha \otimes \beta \otimes \gamma & \quad\:\: _{\rightarrow}\!\!\!\!\!\!\!\!\!\!\!\!\!\!\!\!\!\!^{   \id_1 \otimes  V^{\alpha, \beta}_{\beta+k'} \otimes \id_\gamma } &   1 \otimes  (\beta + k' ) \otimes \gamma  & \quad\:\: _{\rightarrow}\!\!\!\!\!\!\!\!\!\!\!\!\!\!\!\!\!\!^{ \id_1 \otimes   V^{\beta + k', \gamma}_{\beta+k'+1}} &   1 \otimes   (\beta + k'  + l ) & \quad\:\: _{\rightarrow}\!\!\!\!\!\!\!\!\!\!\!\!\!^{ V^{1, \beta+k'+l}_{\beta+k+l}}  &  \beta + k  + l, \\
 C_\beta: 1 \otimes \alpha \otimes \beta \otimes \gamma &\quad\:\: _{\rightarrow}\!\!\!\!\!\!\!\!\!\!\!\!\!\!\!\!\!\!^{   \id_1 \otimes  V^{\alpha, \beta}_{\beta+k'} \otimes \id_\gamma } &   1 \otimes  (\beta + k' ) \otimes \gamma  & \quad\:\: _{\rightarrow}\!\!\!\!\!\!\!\!\!\!\!\!\!\!\!\!\!\!^{      V^{1, \beta + k' }_{\beta+k} \otimes \id_\gamma   } &       (\beta + k  ) \otimes \gamma &\quad\:\: _{\rightarrow}\!\!\!\!\!\!\!\!\!\!\!\!\!^{ V^{ \beta+k, \gamma}_{\beta+k+l}}  &  \beta + k  + l. \\
 \end{array}
 \]
 By the induction hypothesis we have
 \[
 \inf_{z \in \mathbb{T}} \Vert A_\beta - z B_\beta \Vert  \leq C \qdim(\beta)^{-1} \qquad \textrm{ and } \inf_{z \in \mathbb{T}} \Vert B_\beta - z C_\beta \Vert \leq C \qdim(\beta)^{-1},
  \]
  for some constant $C>0$ that only depends on $\alpha$ and $\gamma$. By the triangle inequality
  \[
  \inf_{z \in \mathbb{T}} \Vert A_\beta - z C_\beta \Vert \leq 2C \qdim(\beta)^{-1}.
   \]
   For every $\beta$ let $z_\beta \in \mathbb{T}$ be the phase factor where this infimum is attained so that
   \[
   \Vert A_\beta - z_\beta C_\beta \Vert \leq 2 C \qdim(\beta)^{-1}.
    \]
    By multiplying one of the intertwiners in  the expression of $C_\beta$ with $z_\beta$ we may assume without loss of generality that $z_\beta =1$ for all $\beta$.
   Now consider the following expressions, where $D_\beta$ is obtained from $A_\beta$ by summing over all $k'$ and multiplying with $( V^{1,\alpha}_{\alpha +1}  \otimes \id_{\beta} \otimes \id_\gamma  )^\ast$ on the right. Similarly, $E_\beta$ is obtained from $C_\beta$ by summing over all $k'$ and multiplying with  $( V^{1,\alpha}_{\alpha +1}  \otimes \id_{\beta} \otimes \id_\gamma  )^\ast$. We set,
 \[
 \begin{split}
  D_\beta = & \sum_{k'} V^{1, \beta+k'+l}_{\beta+k+l} (\id_1 \otimes V^{\alpha, \beta+l}_{\beta+k'+l} )  (\id_1 \otimes \id_\alpha \otimes V^{\beta, \gamma}_{\beta+l}) ( V^{1,\alpha}_{\alpha +1}  \otimes \id_{\beta} \otimes \id_\gamma  )^\ast \\
  = & \sum_{k'} V^{1, \beta+k'+l}_{\beta+k+l} (\id_1 \otimes V^{\alpha, \beta+l}_{\beta+k'+l} ) ( V^{1,\alpha}_{\alpha +1}  \otimes \id_{\beta+l}   )^\ast (\id_1 \otimes \id_{\alpha} \otimes V^{\beta, \gamma}_{\beta+l}),\\
  E_\beta = &  \sum_{k'} V^{\beta+k, \gamma}_{\beta + k + l} ( V^{1, \beta+k'}_{\beta+k} \otimes \id_\gamma ) (\id_1 \otimes V^{\alpha, \beta}_{\beta+k'} \otimes \id_{\gamma}  )   ( V^{1,\alpha}_{\alpha +1}  \otimes \id_{\beta} \otimes \id_\gamma  )^\ast.
 \end{split}
 \]
  It follows from the triangle inequality that
  \begin{equation}\label{Eqn=AlmostCommuteSuq2Proof}
  \Vert D_\beta - E_\beta \Vert \leq C K \qdim(\beta)^{-1},
   \end{equation}
   where $K$ is the total number of summands in $D_\beta$ and $E_\beta$ which only depends on $\alpha$ and $\gamma$. But by Lemma \ref{Lem=InterIso} we have for suitable phase factors $z_1, z_2 \in \mathbb{T}$ that
  \[
  \begin{split}
         V^{\alpha+1, \beta +l}_{\beta + k + l} (\id_{\alpha+1} \otimes V^{\beta, \gamma}_{\beta + l}) = & z_1 \sum_{k'} V^{1, \beta+k'+l}_{\beta+k+l} (\id_1 \otimes V^{\alpha, \beta+l}_{\beta+k'+l} )  ( V^{1,\alpha}_{\alpha +1}  \otimes \id_{\beta+l }   )^\ast ( \id_{\alpha+1} \otimes V^{\beta, \gamma}_{\beta+l}), \\
        V^{\beta + k, \gamma}_{\beta + k + l} (V^{\alpha+1, \beta}_{\beta + k} \otimes \id_\gamma)  = & z_2 \sum_{k'} V^{\beta+k, \gamma}_{\beta + k + l} ( V^{1, \beta+k'}_{\beta+k} \otimes \id_\gamma ) (\id_1 \otimes V^{\alpha, \beta}_{\beta+k'} \otimes \id_{\gamma}  )   ( V^{1,\alpha}_{\alpha +1}  \otimes \id_{\beta} \otimes \id_\gamma  )^\ast.
  \end{split}
  \]
 So that \eqref{Eqn=AlmostCommuteSuq2} is just the estimate \eqref{Eqn=AlmostCommuteSuq2Proof}.
 By induction the lemma is proved for any $\alpha \in \mathbb{N}_{\geq 1}$ and $\gamma =1$. Analogously we can do induction on $\gamma$ and the proof follows.
\end{proof}

In conclusion we record the following result.

\begin{thm}\label{Thm=ApproxLinComSUq2}
The QMS defined in Theorem \ref{Thm=Markov} is approximately linear with almost commuting intertwiners.
\end{thm}
\begin{proof}
\eqref{Eqn=AlmostOne}, \eqref{Eqn=AlmostOneB} and \eqref{Eqn=AlmostTwo} follow from Propositions \ref{Prop=ExplicitGamma} and \ref{Prop=AlmostCom}. Finally, by \eqref{Eqn=FimaVergnioux} we see that \eqref{Eqn=AlmostThree} holds for $P$ a linear polynomial.
\end{proof}

\section{Applications to strong solidity: free wreath products and  easy quantum groups}\label{Sect=Consequences}
In this section we gather the consequences of our main results. For the definition of the free wreath product we refer to \cite{Bichon} (and \cite{LemeuxTarrago} for the main properties we need).

\begin{thm}\label{Thm=WreathQMSdot}
Let $\bG$ be a compact quantum group. If $\bG$ carries a QMS of central multipliers that is approximately linear with almost commuting intertwiners then so does the free wreath product $\bG \wr_\ast S_N^+, N \geq 5$. If the QMS on $\bG$ is immediately $L_2$-compact then so is the one on $\bG \wr_\ast S_N^ +$.
\end{thm}
\begin{proof}
By \cite[Theorem 5.11]{LemeuxTarrago} the free wreath product $\bG \wr_\ast S_N^+$ is monoidally equivalent to a compact quantum group $\mathbb{H}$  whose dual $\widehat{\bH}$ is a quantum subgroup of $\widehat{\bG \ast SU_q(2)}$ for $q \in (0,1)$ such that $q + q^{-1} = \sqrt{N}$. By Theorem \ref{Thm=ApproxLinComSUq2} $SU_q(2)$ has a QMS of central multipliers that is approximately linear with almost commuting intertwiners which is moreover immediately $L_2$-compact. Now since approximate linearity with almost commuting intertwiners (and immediate $L_2$-compactness) passes to free products (Theorem \ref{Thm=FreeProduct}), monoidal equivalence (Theorem \ref{Thm=Monoidal}) and dual quantum subgroups (Theorem \ref{Thm=SubGroupdot}) we are done.
\end{proof}

For the definition of the ACPAP we refer to \cite{CFY}.

\begin{thm} \label{Thm=StronglySolidWreath0}
Let $\bG$ be a compact quantum group of Kac type with either the ACPAP or such that $L_\infty(\bG \wr_\ast S_N^+)$ has the W$^\ast$-CBAP. If $\bG$ carries a QMS of central multipliers that is approximately linear with almost commuting intertwiners and which is immediately $L_2$-compact then   the free wreath product $\bG \wr_\ast S_N^+, N \geq 5$ is strongly solid.
\end{thm}
\begin{proof}
If $\bG$ is of Kac type then so is $\bG \wr_\ast S_N^+$.
It follows from \cite[Theorem 6.4 and Remark 6.6]{LemeuxTarrago} that $\bG \wr_\ast S_N^+$ has the W$^\ast$-CBAP. Then we conclude by Theorem \ref{Thm=WreathQMSdot} and Corollary \ref{Cor=KacStronglsySolid}.
\end{proof}

\begin{rmk}
Theorem \ref{Thm=StronglySolidWreath} gives an answer to \cite[Remark 6.6]{LemeuxTarrago}. We note that in \cite[Remark 6.6]{LemeuxTarrago} the strong solidity statement as suggested can only hold under additional assumptions on $\bG$ like the ones in Theorem \ref{Thm=StronglySolidWreath}. Indeed, if there would not but such an assumption then we could consider for instance the case where $\bG$  decomposes as a product of two non-amenable quantum groups whose von Neumann algebras are type II$_1$ factors with the W$^\ast$-CCAP (which exist by \cite{FreslonJFA}). Then $L_\infty(\bG)$ is not strongly solid and neither is the ambient von Neumann algebra $L_\infty(\bG \wr_\ast S_N^+)$.
\end{rmk}

To the knowledge of the author the following result did not appear explicitly in the literature so far. We refer to \cite[Theorem 5.11]{RaumWeber} for strong solidity results for a related series of compact quantum groups. We refer to \cite{BanicaBichonCollins} for the hyperoctahedral series. In the proofs below the symbol $\simeq$ stands for an isomorphism of compact quantum groups (not necessarily preserving the fundamental representation).

\begin{cor}\label{Cor=HnPlus}
The hyperoctahedral compact quantum groups $H_N^{+} \simeq  \mathbb{Z}_2  \wr_\ast S_N^+$  is strongly solid for $N \geq 5$.
\end{cor}
\begin{proof}
This follows directly from Theorem \ref{Thm=StronglySolidWreath0}.
\end{proof}

\begin{thm}\label{Thm=Easy}
The seven series of free orthogonal easy quantum groups that were classified in \cite{WeberAdvances},  \cite{BanicaSpeicher}  under the names $O_{N_3}^+, S_{N_5}^+, H_{N_5}^+, B_{N_4}^+, S_{N_5}'^+, B_{N_4}'^{+}$ and $B_{N_4}^{\# +}$ are strongly solid for $N_3 \geq 3, N_4 \geq 4, N_5  \geq 5$.
\end{thm}
\begin{proof}
It is known that all these examples have the ACPAP (and hence the von Neumann algebras have the W$^\ast$-CCAP) by \cite{CFY}, \cite[Theorem 6.4]{LemeuxTarrago} and the remainder of this proof.

 By \cite[Section 5]{BRV} the quantum group $O_N^+$ is monoidally equivalent to $SU_q(2)$ for $N = q + q^{-1}, q \in (0,1)$ and so we conclude from Theorem \ref{Thm=Markov}, Theorem \ref{Thm=Monoidal} and Corollary \ref{Cor=KacStronglsySolid}. Similarly $S_N^+$ is monoidally equivalent to  $SO_q(3)$ for $N = q^2 + 1 + q^{-2}$; this follows for instance from \cite[Theorem 5.11]{LemeuxTarrago} together with the observation that the dual of $SO_q(3)$ has no quantum subgroups. By \cite[Section 4]{Raum} and \cite[Propositions 5.1 and 5.2]{WeberAdvances} we have identifications as compact quantum groups $S_N'^{+} \simeq S_N^+ \times \mathbb{Z}_2, B_N^+ \simeq O_{N-1}^+, B_N'^{+} \simeq O_{N-1}^+ \times \mathbb{Z}_2,  B_N^{\# +} \simeq O_{N-1}^+ \ast \mathbb{Z}_2$ so that our results follow from the case $O_N^+$ and $S_N^+$ and Theorem \ref{Thm=FreeProduct}. The only remaining case $H_N^+$ was covered in Corollary \ref{Cor=HnPlus}.
\end{proof}

\begin{rmk}
The cases $O_N^+, S_N^+,  B_N^+, S_N'^+, B_N'^{+}$ and $B_N^{\# +}$ in Theorem \ref{Thm=Easy} were covered already in \cite{IsonoIMRN} and \cite{BrannanDocumenta}.
\end{rmk}

We also state the following theorem for completeness, though that here we do not give applications in the non-Kac case. We refer to \cite{CaspersGradient} for such examples. We mention that it is an open problem if a theorem of this form holds under the assumption of the W$^\ast$-CBAP only instead of the W$^\ast$-CCAP.

\begin{thm} \label{Thm=StronglySolidWreath}
Let $\bG$ be a compact quantum group. Suppose that $L_\infty(\bG \wr_\ast S_N^+)$ is solid and has the W$^\ast$-CCAP.  If $\bG$ carries a QMS of central multipliers that is approximately linear with almost commuting intertwiners that is immediately $L_2$-compact then  the free wreath product $\bG \wr_\ast S_N^+, N \geq 5$ is strongly solid.
\end{thm}
\begin{proof}
It follows from \cite[Theorem 6.4 and Remark 6.6]{LemeuxTarrago}, \cite{CFY} that $\bG$ has the W$^\ast$-CCAP. Then we conclude by Theorem \ref{Thm=WreathQMSdot} and Corollary \ref{Cor=NonKacStronglsySolid}.
\end{proof}

\section{Non-commutative Riesz transforms and the Akemann-Ostrand property}\label{Sect=AO}

The aim of this section is to show that the methods in this paper also show that the von Neumann algebras we consider satisfy the Akemann-Ostrand property. The proof is the same as \cite[Section 5]{CIW} but the setting as presented in \cite[Section 5]{CIW} is too narrow for the current setup. Essentially we need to replace the filtrations considered in \cite{CIW} by more general fusion rules. Let us first recall the definition of the Akemann-Ostrand property from \cite{IsonoTrams}.

\begin{dfn}
A von Neumann algebra $M$ satisfies the Akemann-Ostrand property (briefly called ${\rm AO}^{+}$) if there exists a $\sigma$-weakly dense unital C$^\ast$-subalgebra $A \subseteq M$ such that
\begin{enumerate}
\item $A$ is locally reflexive \cite[Section 9]{BrownOzawa};
\item There exists a ucp map $\theta: A \otimes_{{\rm min}} A^{{\rm op}} \rightarrow B(L_2(M))$ such that $\theta(a \otimes b^{{\rm op}}) - a b^{{\rm op}}$ is compact for all $a,b \in A$.
\end{enumerate}
\end{dfn}

Now let $\bG$ be a compact quantum group and let $\Phi$ be a QMS of central multipliers on $\bG$ with generator $\Delta$.

\begin{dfn}\label{Dfn=Subexp}
We say that $\Phi$ has {\it subexponential growth} if $\Delta$ has compact resolvent and for every $\alpha, \gamma \in \Irr(\bG)$ we have
\[
\lim_{\beta \rightarrow \infty} \sup_{
\substack{
\beta' \subseteq \alpha \otimes \beta \otimes \gamma\\
\beta' \in \Irr(\bG) }}\vert \frac{\Delta_{\beta'}}{\Delta_\alpha} - 1 \vert = 0.
\]
Here the limit $\lim_{\alpha \rightarrow \infty} c_\alpha = c$ is defined as saying that for every $\epsilon >0$ there exists a compact set $K \subseteq \Irr(\bG)$ such that for all $\alpha \in \Irr(\bG) \backslash K$ we have $ \vert c_\alpha - c \vert  < \epsilon$.
\end{dfn}

\begin{rmk}\label{Rmk=SubexpAlmost}
The property \eqref{Eqn=AlmostOneB} implies that $\Phi$ has subexponential growth.
\end{rmk}

\begin{rmk}
The subexponential growth condition should be compared to the amenability results from \cite{CiprianiSauvageotAdvances} and \cite[Appendix]{CaspersGradient}. These results show that if the eigenvalues of $\Delta$ grow fast then the von Neumann algebra must be amenable. As a rule of thumb many semi-groups on non-amenable von Neumann algebras will have subexponential growth.
\end{rmk}

The aim of this section is to state the following theorem. By Remark \ref{Rmk=SubexpAlmost} it applies to all QMS's that are approximately linear with almost commuting intertwiners and for which the generator has compact resolvent; in particular it applies to the examples in this paper.

\begin{thm}\label{Thm=AO}
Let $\bG$ be a compact quantum group of Kac type. Let $\Phi$ be a QMS of central multipliers that\footnote{The immedately gradient-$\cS_2$ condition can be replaced by the weaker gradient coarse condition from \cite[Definition 4.1]{CaspersGradient}}  is immediately gradient-$\cS_2$ and which has subexponential growth. Assume that $C_r(\bG)$ is locally reflexive. Then $L_\infty(\bG)$ satisfies ${\rm AO}^+$.
\end{thm}
\begin{proof}[Proof sketch.]
The proof of Theorem \ref{Thm=AO} is a straightforward adaptation of the arguments in \cite[Section 5]{CIW} and  \cite[Theorem 5.13]{CIW} with the following considerations taken into account. The idea is to consider a $L_\infty(\bG)$-$L_\infty(\bG)$-bimodule $H_\nabla$ (called the gradient bimodule or the carr\'e du champ) together with an isometry
\[
S := \partial \Delta^{-\frac{1}{2}}: L_2(\bG) \rightarrow H_\nabla,
\]
 (called the Riesz transform) and we refer to \cite[Eqn. (5.1)]{CIW} for their definitions which make perfect sense in the current context. By \cite[Proposition 3.8 and Proposition 4.4]{CaspersGradient} and using that $\Phi$ is immediately gradient-$\cS_2$ we see that $H_\nabla$ is weakly contained in the coarse bimodule of $L_\infty(\bG)$. We must then prove a suitable replacement of  \cite[Theorem 5.12]{CIW} stating  that for every $x,y \in \Pol(\bG)$ (and hence for every $x,y \in C_r(\bG)$) the map
\begin{equation}\label{Eqn=Compact}
T_{x,y}: L_2(\bG) \rightarrow H_\nabla:  \xi \mapsto S (x  \xi y)  -  xS(\xi)y,
\end{equation}
is compact. Then a standard argument yields condition ${\rm AO}^+$ for which we refer to \cite[Proposition 5.2]{CIW} and finishes the proof.

 The most important part is thus that we must prove that \cite[Theorem 5.12]{CIW} still holds in the current context, meaning that \eqref{Eqn=Compact} is compact. \cite[Theorem 5.12]{CIW} assumes that the von Neumann algebra is filtered (see \cite{CIW}), which is not the case in the setting of Theorem \ref{Thm=AO}. However, we can still make the following observation. For $\alpha \in \Irr(\bG)$ we set the space of matrix coefficients,
\[
\mathcal{A}(\alpha) = \{ (\iota \otimes \omega)(\alpha) \mid  \omega \in M_{n_\alpha}(\mathbb{C})^\ast \}.
\]
Then for $\alpha, \beta \in \Irr(\bG)$ we have that
\[
\mathcal{A}(\alpha) \mathcal{A}(\beta) \subseteq \oplus_{\gamma \subseteq \alpha \otimes \beta} \mathcal{A}(\gamma),
\]
which replaces the filtered condition from \cite{CIW}. With this observation in mind and with the current notion of subexponential growth (Definition \ref{Dfn=Subexp}), the proof of \cite[Theorem 5.12]{CIW} translates  literally to the current setting.
\end{proof}

Essentially the theorem applies to all the examples mentioned in Section \ref{Sect=Consequences}. For instance we get the following result. Except for the case of $H_N^+$ this was already known from \cite{IsonoIMRN}, \cite{IsonoTrams}, \cite{BrannanDocumenta}.

\begin{thm}\label{Thm=EasyAO}

The seven series of free orthogonal easy quantum groups that were classified in \cite{WeberAdvances},  \cite{BanicaSpeicher}  under the names $O_{N_3}^+, S_{N_5}^+, H_{N_5}^+, B_{N_4}^+, S_{N_5}'^+, B_{N_4}'^{+}$ and $B_{N_4}^{\# +}$ satisfy AO$^+$ for $N_3 \geq 3, N_4 \geq 4, N_5  \geq 5$.
\end{thm}

 Finally it should be mentioned that also through condition AO$^+$ strong solidity results can be obtained through the results from \cite{IsonoTrams} and \cite{PopaVaesCrelle}.

\end{document}